\documentclass[11pt]{article}

\usepackage{amsfonts,amssymb,amsmath,amsthm,epsfig,euscript}

\newtheorem{theorem}{Theorem}
\newtheorem{lemma}[theorem]{Lemma}
\newtheorem{corollary}[theorem]{Corollary}

\long\def\symbolfootnote[#1]#2{\begingroup
\def\thefootnote{\fnsymbol{footnote}}\footnote[#1]{#2}\endgroup}

\newcommand{\des}[1]{\mathrm{des}(#1)}
\newcommand{\cdes}[1]{\mathrm{cdes}(#1)}
\newcommand{\cyc}[1]{\mathrm{cyc}(#1)}
\newcommand{\LtRMin}[1]{\mathrm{LtRMin}(#1)}

\newcommand{\sg}{\sigma}

\newcommand{\cref}[1]{Corollary \ref{corollary:#1}}

\newcommand{\red}[1]{\mathrm{red}(#1)}

\newcommand{\tmch}[1]{\text{$\tau$-$\mathrm{mch}$}(#1)}
\newcommand{\ctmch}[1]{\text{$c$-$\tau$-$\mathrm{mch}$}(#1)}

\newcommand{\cumch}[1]{\text{$c$-$\Upsilon$-$\mathrm{mch}$}(#1)}

\newcommand{\Umch}[1]{\text{$\Upsilon$-$\mathrm{mch}$}(#1)}

\title{Pattern Matching in the Cycle Structures of Permutations.}

\author{
Miles Eli Jones\\
\small Department of Mathematics\\[-0.8ex]
\small University of California, San Diego\\[-0.8ex]
\small La Jolla, CA 92093-0112. USA\\[-0.8ex]
\small \texttt{mej005@math.ucsd.edu}
\and
Jeffrey Remmel\footnote{Partially supported by NSF grant DMS 0654060.} \\
\small Department of Mathematics\\[-0.8ex]
\small University of California, San Diego\\[-0.8ex]
\small La Jolla, CA 92093-0112. USA\\[-0.8ex]
\small \texttt{jremmel@ucsd.edu}
\and
}

\date{\small Submitted: Date 1;  Accepted: Date 2;
 Published: Date 3.\\
\small MR Subject Classifications: 05A15, 05E05}

\begin{document}
\maketitle

\begin{abstract}
\noindent

In this paper, we study the occurrence of patterns in the 
cycle structures of permutations. 

\end{abstract}

\section{Introduction}

The notion
of  patterns in permutations and words
has proved to be a useful language in a
variety of seemingly unrelated problems including the theory of
Kazhdan-Lusztig polynomials, singularities of Schubert
varieties, Chebyshev polynomials, rook polynomials
for Ferrers boards, and various sorting algorithm 
including sorting stacks and sortable permutations.
The study of occurrences of 
patterns in words and permutations is a new, but
rapidly growing, branch of combinatorics which has its roots in the works
by Rotem, Rogers, and Knuth in the 1970s and early 1980s. 
The first systematic study of permutation patterns 
was not undertaken until the paper by
Simion and Schmidt~\cite{SchSim} which appeared in 1985. The
field has experienced explosive
growth since 1992.

The goal of this paper is to initiate the 
study pattern matching conditions in 
the cycle structure of a permutation. First we recall the basic 
definitions for pattern matching in permutations. 
Given a sequence $\sg = \sg_1 \ldots \sg_n$ of distinct integers,
let $\red{\sg}$ be the permutation found by replacing the
$i^{\textrm{th}}$ largest integer that appears in $\sg$ by $i$.  For
example, if $\sg = 2~7~5~4$, then $\red{\sg} = 1~4~3~2$.  Given a
permutation $\tau=\tau_1 \ldots \tau_j$ in the symmetric group $S_j$, we say a
permutation $\sg = \sg_1 \ldots \sg_n \in S_n$ has  a {\em
$\tau$-match starting at position $i$} provided $\red{\sg_i \ldots \sg_{i+j-1}}
= \tau$.  Let $\tmch{\sg}$ be the number of $\tau$-matches in the
permutation $\sg$.  Similarly, we say that $\tau$ {\em occurs} in
$\sg$ if there exist $1 \leq i_1 < \cdots < i_j \leq n$ such that
$\red{\sg_{i_1} \cdots \sg_{i_j}} = \tau$.  We say that $\sg$ {\em
avoids} $\tau$ if there are no occurrences of $\tau$ in $\sg$. 

These definitions can naturally be extended to sets of permutations. 
That is, given a set of permutations 
$\Upsilon$ in the symmetric group $S_j$, define a
permutation $\sg = \sg_1 \ldots \sg_n \in S_n$ to have a {\em
$\Upsilon$-match starting at position $i$} 
provided $\red{\sg_i \ldots \sg_{i+j-1}}
\in \Upsilon$.  Let $\Umch{\sg}$ be the number of $\Upsilon$-matches in the
permutation $\sg$.  Similarly, we say that $\Upsilon$ {\em occurs} in
$\sg$ if there exist $1 \leq i_1 < \cdots < i_j \leq n$ such that
$\red{\sg_{i_1} \cdots \sg_{i_j}} \in  \Upsilon$.  We say that $\sg$ {\em
avoids} $\Upsilon$ if there are no occurrences of $\Upsilon$ in $\sg$.

In this paper, we want to study matching conditions within 
the cycle structure of a permutation. Suppose that 
$\tau=\tau_1 \ldots \tau_j$ is a permutation in $S_j$ and 
$\sg$ is a permutation in $S_n$ with $k$ cycles $C_1 \ldots C_k$. 
We shall always write cycles in the form 
$C_i =(c_{0,i}, \ldots, c_{p_i-1,i})$  where $c_{0,i}$ is the smallest 
element in $C_i$ and $p_i$ is the length of $C_i$ and we arrange 
the cycles by increasing smallest elements. That is, we arrange 
the  cycles of $\sg$ so that $c_{0,1} < \cdots < c_{0,k}$. Then 
we say that $\sg$ has a {\em cycle $\tau$-match} ($c$-$\tau$-match)
if there is an $i$ such that $C_i =(c_{0,i}, \ldots, c_{p_i-1,i})$ where 
$p_i \geq j$ and an $r$ such that 
$\red{c_{r,i} c_{r+1,i} \ldots c_{r+j-1,i}} = 
\tau$ where we take indices of the form $r+s$ modulo $p_i$.   
Let $\ctmch{\sg}$ be the number of cycle $\tau$-matches in the
permutation $\sg$.  For example, 
if $\tau =2~1~3$ and $\sg = (1,10,9)(2,3)(4,7,5,8,6)$, then 
$9~1~10$ is a cycle $\tau$-match in the first cycle and 
$7~5~8$ and $6~4~7$ are cycle $\tau$-matches in the third cycle so 
that $\ctmch{\sg} =3$.  Similarly, we say that $\tau$ {\em cycle occurs} in
$\sg$ if there exists an $i$ such that 
$C_i =(c_{0,i}, \ldots, c_{p_i-1,i})$ where 
$p_i \geq j$ and there is 
an $r$ with $0 \leq r \leq p_i-1$ and indices 
$0 \leq i_1 < \cdots < i_{j-1} \leq p_i-1$ such that 
$\red{c_{r,i} c_{r+i_1,i} \ldots c_{r+i_{j-1},i}} = \tau$ where 
the indices $r+{i_s}$ are taken mod $p_i$.  
We say that $\sg$ {\em
cycle avoids} $\tau$ if there are no cycle occurrences of $\tau$ in $\sg$. 
For example, 
if $\tau =1~2~3$ and $\sg = (1,10,9)(2,3)(4,8,5,7,6)$, then 
$4~5~7$, $4~5~6$,  and $5~6~8$ are cycle occurrences of $\tau$  
in the third cycle.

We can 
extend of the notion of cycle matches and cycle occurrences to sets of 
permutations in the obvious fashion. That is, suppose that 
$\Upsilon$ is a set of permutations in $S_j$ and 
$\sg$ is a permutation in $S_n$ with $k$ cycles $C_1 \ldots C_k$. 
 Then 
we say that $\sg$ has a {\em cycle $\Upsilon$-match} ($c$-$\Upsilon$-match)
if there is an $i$ such that $C_i =(c_{0,i}, \ldots, c_{p_i-1,i})$ where 
$p_i \geq j$ and an $r$ such that $\red{c_{r,i} \ldots c_{r+j-1,i}} \in 
\Upsilon$ where we take indices of the form $r+s$ modulo $p_i$.  
Let $\cumch{\sg}$ be the number of cycle $\Upsilon$-matches in the
permutation $\sg$.    Similarly, we say that $\Upsilon$ {\em cycle occurs} in
$\sg$ if there exists an $i$ 
such that $C_i =(c_{0,i}, \ldots, c_{p_i-1,i})$ where 
$p_i \geq j$ and there is 
an $r$ with $0 \leq r \leq p_i-1$ and indices 
$0 \leq i_1 < \cdots < i_{j-1} \leq p_i-1$ such that $\red{c_{r,i} c_{r+i_1,i} \ldots c_{r+i_{j-1},i}} \in \Upsilon$ 
where the indices $r+{i_s}$ are taken mod $p_i$.  
We say that $\sg$ {\em
cycle avoids} $\Upsilon$ if there are no cycle occurrences of 
$\Upsilon$ in $\sg$.

Given $\Upsilon \subseteq S_j$, we
 let $\mathcal{AS}_n(\Upsilon)$ ($\mathcal{CAS}_n(\Upsilon)$)
denote the set of permutations of $S_n$ which 
avoid (cycle avoid) $\Upsilon$ and 
$aS_n(\Upsilon) =|\mathcal{AS}_n(\Upsilon)|$  
($caS_n(\Upsilon) =|\mathcal{CAS}_n(\Upsilon)|$). 
Similarly, we
 let $\mathcal{NMS}_n(\Upsilon)$ ($\mathcal{NCMS}_n(\Upsilon)$)
denote the set of permutations of $S_n$ which 
have no $\Upsilon$-matches (no cycle $\Upsilon$-matches) $\Upsilon$ and 
$nmS_n(\Upsilon) =|\mathcal{NMS}_n(\Upsilon)|$  
($ncmS_n(\Upsilon) =|\mathcal{NCMS}_n(\Upsilon)|$). 
Throughout this paper, when $\Upsilon =\{\tau\}$ is a singleton, 
we shall just write the $\tau$ rather than $\{\tau\}$. Thus for example, 
we shall write $\mathcal{AS}_n(\tau)$ for 
$\mathcal{AS}_n(\Upsilon)$ when $\Upsilon =\{\tau\}$.

Given $\alpha$ and $\beta$ in 
$S_j$, we say that $\alpha$ and $\beta$ are 
{\em Wilf equivalent}  if $aS_n(\alpha) = aS_n(\beta)$ for all $n$. 
We say that $\alpha$ and $\beta$ are 
{\em matching Wilf equivalent} (m-Wilf equivalent)  
if $nmS_n(\alpha) = nmS_n(\beta)$ for all $n$.
For any permutation $\sg = \sg_1 \ldots \sg_n$, we let $\sg^r$ be 
the reverse of $\sg$ and $\sg^c$ be the complement of $\sg$. That is, 
$\sg^r = \sg_n \ldots \sg_1$ and $\sg^c = (n+1 - \sg_1) \ldots (n+1 - \sg_n)$. 
It is well known that Wilf equivalence classes and m-Wilf equivalence 
classes are closed under reverse and complementation. We say that $\alpha$ and $\beta$ are 
{\em cycle avoidance Wilf equivalent} (ca-Wilf equivalent) if $caS_n(\alpha) = caS_n(\beta)$ for all $n$ and we say that $\alpha$ and $\beta$ are 
{\em cycle matching Wilf equivalent} (cm-Wilf equivalent) if $ncmS_n(\alpha) = ncmS_n(\beta)$ for all $n$. If $\alpha$ and 
$\beta$ are cycle avoidance Wilf equivalent, we shall write 
$\alpha \sim_{ca} \beta$. If $\alpha$ and 
$\beta$ are cycle matching Wilf equivalent, we shall write 
$\alpha \sim_{cm} \beta$. Similarly, for sets of 
permutations $\Gamma$ and $\Delta$ in $S_j$, 
we say that $\Gamma$ and $\Delta$ are cycle avoidance Wilf equivalent (ca-Wilf equivalent) if $caS_n(\Gamma) = caS_n(\Delta)$ for all $n$ and we say that $\Gamma$ and $\Delta$ are cycle matching Wilf equivalent (cm-Wilf equivalent) if $ncmS_n(\Gamma) = ncmS_n(\Delta)$ for all $n$.

If $\sg$ is a permutation in $S_n$ with $k$ cycles $C_1 \ldots C_k$, 
then we let the {\em cycle reverse of $\sg$}, denoted by $\sg^{cr}$, be the 
permutation which arises from $\sg$ by replacing each 
cycle $C_i =(c_{0,i}, c_{1,i},\ldots, c_{p_i-1,i})$ by its reverse cycle 
$C_i^{cr} =(c_{0,i},c_{p_i-1,i}, \ldots c_{1,i})$. For example, 
if $\sg = (1,10,9)(2,3)(4,7,5,8,6)$, then 
$\sg^{cr} = (1,9,10)(2,3)(4,6,8,5,7)$.  We let the 
cycle complement of $\sg$, denoted by $\sg^{cc}$, be the permutation 
that results from $\sg$ by  
replacing each element $i$ in the cycle structure of $\sg$ by 
$n+1 -i$.  For example, 
if $\sg = (1,10,9)(2,3)(4,7,5,8,6)$, then 
$\sg^{cr} = (10,1,2)(9,8)(7,4,6,3,5) = (1,2,10)(3,5,7,4,6)(8,9)$. 
Note that in general $\sg^r$, $\sg^c$, $\sg^{cr}$ and $\sg^{cc}$ 
are all distinct. For example, $\sg = 2~3~1~5~4$ so that it 
cycle structure is $(1,2,3)(4,5)$, then 
\begin{eqnarray*}
\sg^r &=& 4~5~1~3~2, \\
\sg^c &=& 4~3~5~1~2, \\
\sg^{cr} &=& (1,3,2)(4,5) = 3~1~2~5~4, \ \mbox{and} \\
\sg^{cc} &=& (5,4,3)(2,1) = 2~1~5~3~4.
\end{eqnarray*}

It is easy to see that for any permutation $\sg \in S_n$,
\begin{enumerate} 
\item $\sg$ has a cycle $\tau$-match if and only if $\sg^{cr}$ has 
a cycle $\tau^r$-match, 
\item $\sg$ has a cycle $\tau$-match if and only if $\sg^{cc}$ has 
a cycle $\tau^c$-match, 
\item $\sg$ has a cycle $\tau$ occurrence  if and only if $\sg^{cr}$ has 
a cycle $\tau^r$ occurrence, and 
\item $\sg$ has a cycle $\tau$ occurrence  if and only if $\sg^{cc}$ has 
a cycle $\tau^c$ occurrence. 
\end{enumerate}
It then easily follows that for all permutations $\tau$, 
$ncmS_n(\tau) = ncmS_n(\tau^r) = ncmS_n(\tau^c)$ 
so that $\tau$, $\tau^r$, and $\tau^c$ are all cycle matching  
Wilf equivalent. Similarly,  $caS_n(\tau) = caS_n(\tau^r) = caS_n(\tau^c)$ 
so that $\tau$, $\tau^r$, and $\tau^c$ are all cycle avoidance  
Wilf equivalent.  Finally we observe that our definitions also ensure 
that for any $\tau = \tau_1 \ldots \tau_j \in S_j$, any cyclic 
rearrangement of $\tau$, $\tau^{(i)} = \tau_i \ldots \tau_j \tau_1 \ldots 
\tau_{i-1}$ also has the property that for any $\sg \in S_n$, $\tau$ 
cycle occurs in $\sg$ if and only if $\tau^{(i)}$ cycle occurs in 
$\sg$. Thus for all $1 \leq j$, 
$caS_n(\tau) = caS_n(\tau^{(i)})$ so that $\tau$ and $\tau^{(i)}$ 
are cycle avoidance Wilf equivalent.

Given a permutation $\sg = \sg_1 \ldots \sg_n \in S_n$, we let 
$\des{\sg} = |\{i: \sg_i > \sg_{i+1}\}|$.  We say that 
$\sg_j$ is a {\em left-to-right minima} of $\sg$ if $\sg_j < \sg_i$ for 
all $i<j$. We let $\LtRMin{\sg}$ denote the number of 
left-to-right minma of $\sg$. Given a cycle 
$C=(c_0, \ldots, c_{p-1})$ where $c_0$ is the smallest element in 
the cycle, we let $\cdes{C} = 1+ \des{c_0 \ldots c_{p-1}}$. Thus 
$\cdes{C}$ counts the number of descent pairs as we traverse 
once around the cycle because the extra factor of $1$ counts 
the descent pair $c_{p-1}>c_0$. For example if $C = (1,5,3,7,2)$, then 
$\cdes{C} = 3$ which counts the descent pairs $53$, $72$, and $21$ as 
we traverse once around $C$.  By convention, if 
$C =(c_0)$ is one-cycle, we let $\cdes{C} = 1$. If $\sg$ is a permutation in $S_n$ with $k$ cycles $C_1 \ldots C_k$, then 
we define $\cdes{\sg} = \sum_{i=1}^k \cdes{C_i}$.  We let 
$\cyc{\sg}$ denote the number of cycles of $\sg$.

The main goal of this paper is to study the generating functions 
\begin{equation}\label{caU}
CA_{\Upsilon}(t) = 1 + \sum_{n \geq 1} caS_n(\Upsilon) \frac{t^n}{n!},
\end{equation} 
and 
\begin{equation}\label{ncmU}
NCM_{\Upsilon}(t) = 1 + \sum_{n \geq 1} ncmS_n(\Upsilon) \frac{t^n}{n!}
\end{equation}
for $\Upsilon \subseteq S_j$ as well as refinements of such 
generating functions such as 
\begin{eqnarray*}
CA_{\Upsilon}(t,x) &=& 1 + \sum_{n \geq 1}  \frac{t^n}{n!} 
\sum_{\sg \in \mathcal{CAS}_n(\Upsilon)} x^{\cyc{\sg}},\\
CA_{\Upsilon}(t,x,y) &=& 1 + \sum_{n \geq 1}  \frac{t^n}{n!} 
\sum_{\sg \in \mathcal{CAS}_n(\Upsilon)} x^{\cyc{\sg}} y^{\cdes{\sg}},\\
NCM_{\Upsilon}(t,x) &=& 1 + \sum_{n \geq 1}  \frac{t^n}{n!} 
\sum_{\sg \in \mathcal{NCMS}_n(\Upsilon)} x^{\cyc{\sg}}, \ \mbox{and}\\
NCM_{\Upsilon}(t,x,y) &=& 1 + \sum_{n \geq 1}  \frac{t^n}{n!} 
\sum_{\sg \in \mathcal{NCMS}_n(\Upsilon)} x^{\cyc{\sg}} y^{\cdes{\sg}}.
\end{eqnarray*}

We know of several ways to approach the this problem. 
The most direct way 
is to use the theory of exponential structures to 
reduce the problem down to studying pattern matching in 
$n$-cycles. That is, let $\mathcal{L}_m$ denote the set 
$m$-cycles in $S_m$. Suppose that $R$ is a ring and 
for each $m \geq 1$, we have a weight function 
$W_m:\mathcal{L}_m \rightarrow R$. We let 
$W(L_m) = \sum_{C \in \mathcal{L}_m} W_m(C)$.  Now suppose that 
$\sg \in S_n$ and the cycles of $\sg$ are $C_1, \ldots , C_k$. 
If $C_i$ is of size $m$, then we let $W(C_i) =W_m(red(C_i))$ where 
$red(C_i)$ is the $m$-cycle in $S_m$ that results by replacing 
$j$-th smallest element in $C_i$ by $j$ for $j=1, \ldots, m$. For example, 
if $C_i =(1,5,7,10,4)$, then $red(C_i) = (1,3,4,5,2)$. 
Then we define the weight of $\sg$, $W(\sg)$, by
$$W(\sg)= \prod_{i=1}^k W(C_i).$$   
We let $\mathcal{C}_{n,k}$ denote the set of all permutations 
of $S_n$ with $k$ cycles. 
This given, the following theorem easily follows 
from the theory of exponential structures, see \cite{Stan}. 

\begin{theorem}\label{genexp}
\begin{equation}
1+\sum_{n=1}^\infty \frac{t^n}{n!} \sum_{k=1}^n x^k 
\sum_{\sg \in \mathcal{C}_{n,k}} W(\sg) = 
e^{x \sum_{m \geq 1} \frac{W(L_m)t^m}{m!}}.
\end{equation}
\end{theorem}

Let $\Upsilon \subseteq S_j$. Then we will be most interested in the 
special case of weight 
functions $W_m$ where $W_m(C) =1$ if $C$ cycle avoids 
a set of permutations  and $W_m(C) =0$ otherwise 
or where $W_m(C) =1$ if $C$ has no cycle $\Upsilon$-matches and $W_m(C) =0$ otherwise.  We  let  
$\mathcal{CAS}_{n,k}(\Upsilon)$ denote the set of permutations $\sg$ 
of $S_n$ with $k$ cycles such that $\sg$ cycle avoids $\Upsilon$ and 
we let $caS_{n,k}(\Upsilon) = |\mathcal{CAS}_{n,k}(\Upsilon)|$. 
We  let $\mathcal{NCMS}_{n,k}(\Upsilon)$ denote the set  of permutations $\sg$ 
of $S_n$ with $k$ cycles such that $\sg$ has no cycle $\Upsilon$-matches 
and $ncmS_{n,k}(\Upsilon) = |\mathcal{NCMS}_{n,k}(\Upsilon)|$. 
Similarly, we let $\mathcal{L}_m^{ca}(\Upsilon)$ be 
the set of $m$ cycles $\gamma$ 
in $S_m$ such $\gamma$ cycle avoids $\Upsilon$, $L_{m}^{ca}(\Upsilon)  = 
|\mathcal{L}_m^{ca}(\Upsilon)|$, 
$\mathcal{L}_m^{ncm}(\Upsilon)$ denote the set  of $m$ cycles $\gamma$ 
in $S_m$ such $\gamma$ has no cycle $\Upsilon$-matches, and 
$L_{m}^{ncm}(\Upsilon)  = |\mathcal{L}_m^{ncm}(\Upsilon)|$.
Then a special case of Theorem \ref{genexp} is the following theorem. 

\begin{theorem}\label{Fundrefined} 
 \begin{equation}\label{expcaUtx}
CA_{\Upsilon} (t,x) = 1 + \sum_{n \geq 1} \frac{t^n}{n!} \sum_{k=1}^n 
caS_{n,k}(\Upsilon) x^k = 
e^{x \sum_{m \geq 1} L_m^{ca}(\Upsilon) \frac{t^m}{m!}},
\end{equation}
\begin{equation}\label{expncmUtx}
NCM_{\Upsilon} (t,x) =  1 + \sum_{n \geq 1} \frac{t^n}{n!} \sum_{k=1}^n 
ncmS_{n,k}(\Upsilon) x^k  =  e^{x \sum_{m \geq 1} L_m^{ncm}(\Upsilon) \frac{t^m}{m!}},
\end{equation}
\begin{equation}\label{expcaUtxy}
CA_{\Upsilon} (t,x,y) = 1 + \sum_{n \geq 1} \frac{t^n}{n!} \sum_{k=1}^n 
x^k \sum_{\sg \in \mathcal{CAS}_{n,k}(\Upsilon)} y^{\cdes{\sg}}= 
e^{x \sum_{m \geq 1}  \frac{t^m}{m!}  
\sum_{C \in \mathcal{L}^{ca}_m(\Upsilon)} y^{\cdes{C}}},
\end{equation}
and 
\begin{equation}\label{expncmUtxy}
NCM_{\Upsilon} (t,x,y) = 1 + \sum_{n \geq 1} \frac{t^n}{n!} \sum_{k=1}^n 
x^k \sum_{\sg \in \mathcal{NCMS}_{n,k}(\Upsilon)} y^{\cdes{\sg}}= 
e^{x \sum_{m \geq 1}  \frac{t^m}{m!} 
\sum_{C \in \mathcal{L}^{ncm}_m(\Upsilon)} y^{\cdes{C}}}.
\end{equation}
\end{theorem}

For example, suppose that $\tau = 1~2$.  It is clear 
that any cycle of length $k$ where $k \geq 2$ has both a 
cycle occurrence of $\tau$ and 
a cycle $\tau$-match so that $L_m^{ca}(12) = L_m^{ncm}(12) = 0$ 
if $m \geq 2$.  Since 1-cycles can  not have any cycle occurrences 
of $\tau$ or any cycle $\tau$-matches by definition, it 
follows that
\begin{equation*} 
y = \sum_{C \in \mathcal{L}^{ca}_1(12)} y^{\cdes{C}} \\
= \sum_{C \in \mathcal{L}^{ncm}_1(12)} y^{\cdes{C}}.
\end{equation*}
Thus 
\begin{equation*}
CA_{12}(t,x,y) = NCM_{12}(t,x,y) = e^{xyt}.
\end{equation*}
Next consider $\tau = 1~2~3$. It is easy to see that 
for $k \geq 3$, the only $k$-cycle which cycle avoids  
$\tau$ is the cycle $(1,k,k-1, \ldots, 2)$. Let 
\begin{equation*}
A_m(y) = \sum_{C \in \mathcal{L}^{ca}_m(123)} y^{\cdes{C}},
\end{equation*} 
then clearly $A_1(y) = y$ since $\cdes{(1)} =1$, 
$A_2(y) = y$ since $\cdes{(1,2)} =1$, and for 
$k \geq 3$, $A_k(y) = y^{k-1}$ since 
$\cdes{(1,k,\ldots,2)} = k-1$. Thus 
\begin{equation*}
CA_{123}(t,x,y) = e^{x \left(yt+\sum_{m \geq 2} \frac{y^{m-1}t^m}{m!}\right)} = 
e^{x\left(yt+ \frac{1}{y}(e^{yt}-1-yt)\right)}.
\end{equation*}

It turns out that if $\tau \in S_j$ is a permutation that 
starts with 1, then we can reduce the problem of finding 
$NCM_{\tau}(t,x)$ and $NCM_{\tau}(t,x,y)$ to the usual problem of finding 
the generating function of permutations that have no $\tau$-matches. 
That is, suppose we are given a permutation $\sg \in S_n$ with 
$k$-cycles $C_1 \cdots  C_k$. Assume we have arranged 
the cycles so that the smallest element in each cycle is on 
the left and we arrange the cycles by decreasing smallest elements. Then 
we let $\bar{\sg}$ be the permutation that arise from 
$C_1 \cdots C_k$ by erasing all the parenthesis and commas. 
For example, if 
$\sg = (7,10,9,11)\ (4,8,6) \ (1,5,3,2)$, then 
$\bar{\sg} = 7~10~9~11~4~8~6~1~5~3~2$. It is easy to see that 
the minimal elements of the cycles correspond to left-to-right minima in 
$\bar{\sg}$.  It is also easy to see 
that under our bijection $\sg \rightarrow \bar{\sg}$, that 
$cdes(\sg) = \des{\bar{\sg}}+1$ since every left-to-right minima 
is part of a descent pair in $\bar{\sg}$. For example, if 
$\sg = (7,10,9,11)\ (4,8,6) \ (1,5,3,2)$ so that 
$\bar{\sg} = 7~10~9~11~4~8~6~1~5~3~2$, 
$\cdes{(7,10,9,11)} = 2$, $\cdes{(4,8,6)} =2$, and 
$\cdes{(1,5,3,2)} =4$ so that $\cdes{\sg} = 2+2+4 =8$ while 
$\des{\bar{\sg}} =  7$.  This given, we 
have the following lemma. 

\begin{lemma}\label{key1}
If $\tau \in S_j$ and $\tau$ starts with 1, then 
for any $\sg \in S_n$, 
\begin{enumerate}
\item $\sg$ has $k$ cycles if and only if $\bar{\sg}$ has $k$ left-to-right
minima,
\item $\cdes{\sg} = 1+\des{\bar{\sg}}$, and 
\item $\sg$ has no cycle-$\tau$-matches if and 
only if $\bar{\sg}$ has no $\tau$-matches.
\end{enumerate} 
\end{lemma}
\begin{proof}
  
For (3), suppose that $\bar{\sg} =\bar{\sg}_1 \ldots \bar{\sg_n}$ and 
$\bar{\sg}_i =1$.  Since $\tau$ starts with $1$, it is easy to see that 
any $\tau$-match in $\bar{\sg}$ must either occur weakly to 
the right of $\bar{\sg}_i$ or strictly to left of $\bar{\sg}_i$. That is, 
1 can be part of $\tau$-match in $\bar{\sg}$ only if the $\tau$-match 
starts at position $i$.  If a $\tau$-match occurred weakly to the right of $\bar{\sg}_i$, then that $\tau$-match would correspond to a cycle-$\tau$-match in $C_k$ in $\sg$. 

Next suppose that the $\tau$-match occurred strictly to the left of 
$\bar{\sg}_i =1$. Then we claim that we can make a similar 
argument with respect to the cycles $C_1 \cdots C_{k-1}$. That is, 
suppose that $C_{k-1}$ starts with $m$. Then $m$ must be the smallest 
element among $\bar{\sg}_1 \ldots \bar{\sg}_{j-1}$. Suppose that 
$\bar{\sg}_s =m$ where $1 \leq s < j$. Then again we can argue 
that any $\tau$-match in $\bar{\sg}_1 \ldots \bar{\sg}_{j-1}$ must occur 
either weakly to the right of $\bar{\sg}_s$ or strictly to left of 
$\bar{\sg}_s$. If the $\tau$-match in $\bar{\sg}_1 \ldots \bar{\sg}_{j-1}$ occurs weakly to the right of $\bar{\sg}_s$, then it would correspond to a 
 cycle-$\tau$-match 
in $C_{k-1}$.  Continuing on in this way, we see that any $\tau$-match 
in $\bar{\sg}$ must correspond to a cycle $\tau$-match in $C_i$ for 
some $i$. 

Vice versa, it is easy to see that since $\tau$ starts with $1$, 
the only way that a cycle-$\tau$-match in $C_i$ can involve 
the smallest element $c_{0,i}$ in the cycle $C_i$ is if $c_{0,i}$ corresponds 
to the 1 in  $\tau$ in cycle match. But this easily implies that any 
$\tau$-cycle match in $C_i$ must also correspond to a $\tau$-match in the 
elements of $\bar{\sg}$ corresponding to $C_i$. 

Thus we have proved that for any $\sg$, $\sg$ has cycle-$\tau$-match 
if only if $\bar{\sg}$ has a $\tau$-match.  
\end{proof}

We should note that if a permutation $\tau$ does not start 
with 1, then it may be that case that $ncmS_n(\tau) \neq nmS_n(\tau)$. 
For example, $\tau = ~3~1~4~2$ is the smallest permutation such 
that neither $\tau$, $\tau^r$, $\tau^c$, nor $(\tau^r)^c$ starts 
with one. For example, even though we do not know how to 
compute closed forms for 
$NCM(t)$ and $NM(t)$, we have computed the following table.

\begin{center}
\begin{tabular}{|c|c|c|c|} 
\hline 
$n$ & $L_n^{ncm}(3142)$ & $NCM_n(3142)$ & $NM_n(3142)$\\
\hline
1& 1 & 1 & 1\\
\hline
2 & 1 & 2 & 2\\
\hline
3 & 2 & 6 & 6 \\
\hline
4 & 5 & 23 & 23 \\
\hline
5 & 20  & 110 & 110 \\
\hline
6 & 92 & 632 & 632 \\
\hline
7 & 532 & 4236 & 4237 \\
\hline  
8  & 3565 & 32448  & 32465\\
\hline
\end{tabular}
\end{center}

One consequence of Lemma \ref{key1} is that we can automatically 
obtain refinements of generating functions for 
the number of permutations with no $\tau$-matches when $\tau$ starts 
with 1. That is, let 
\begin{eqnarray*}
NM_{\tau}(t,x) &=& \sum_{n \geq 0} \frac{t^n}{n!} 
\sum_{\sg \in \mathcal{NMS}_n(\tau)} x^{\LtRMin{\sg}} \ \mbox{and} \\
NM_{\tau}(t,x,y) &=& \sum_{n \geq 0} \frac{t^n}{n!} 
\sum_{\sg \in \mathcal{NMS}_n(\tau)} x^{\LtRMin{\sg}} y^{1+\des{\sg}}.
\end{eqnarray*}
Then we have the following corollary of Lemma \ref{key1}.
\begin{corollary} \label{cor:starts1}
If $\tau \in S_j$ and $\tau$ starts with 1, then 
\begin{eqnarray}\label{NM=NCMtxy}
NCM_{\tau}(t,x) &=& NM_{\tau}(t,x) \ \mbox{and} \\
NCM_{\tau}(t,x,y) &=& NM_{\tau}(t,x,y).
\end{eqnarray}
\end{corollary}

Then by Theorem \ref{Fundrefined} and Lemma \ref{key1}, if 
$\tau \in S_j$ and $\tau$ starts with 1, we have 
that 
\begin{eqnarray*} 
NM_{\tau}(t,1)  &=& \sum_{n \geq 0} NM_n(\tau) \frac{t^n}{n!} \\
&=& NCM(t,1) \\
&=& e^{\sum_{m \geq 1} L_m^{ncm}(\tau) \frac{t^m}{m!}}
\end{eqnarray*}
so that  
\begin{equation}\label{eq:1231}
ln(NM_{\tau}(t,1)) = \sum_{m \geq 1} L_m^{ncm}(\tau) \frac{t^m}{m!}.
\end{equation}
But then 
\begin{eqnarray}\label{eq:1232}
NM(t,x) &=& NCM(t,x) \\
 &=& e^{x\sum_{m \geq 1} L_m^{ncm}(\tau) \frac{t^m}{m!}} \nonumber \\
&=& e^{x \ln(NM_{\tau}(t,1))} = (NM_{\tau}(t,1))^x
\end{eqnarray}
Thus if we can compute  $NM_{\tau}(t,1)$ for a permutation $\tau \in S_j$ that 
starts with 1, we automatically can compute $NM_{\tau}(t,x)$.
For example, 
Goulden and Jackson \cite{GJ} proved
that when $\tau = 1~2 \ldots k$, then
\begin{equation}\label{eq:GJ}
NM_\tau(t)= \frac{1}{\sum_{i \geq 0} \frac{t^{ki}}{(ki)!} -
\frac{t^{ki+1}}{(ki+1)!}}.
\end{equation}
Hence, we automatically have the following refinement of Goulden 
and Jackson's result. 
\begin{theorem} 
If  $\tau = 12\ldots k$ where $k \geq 2$, then  
\begin{equation}
NM_{j \ldots 2~1}(t,x) = \left( \frac{1}{\sum_{i \geq 0} \frac{t^{ki}}{(ki)!} -
\frac{t^{ki+1}}{(ki+1)!}}\right)^x.
\end{equation}
\end{theorem}

An example, where one can use the full power of Theorem \ref{genexp} is 
the following.   In section 2, we shall show that 
\begin{equation}\label{eq:132}
\sum_{n \geq 1} \frac{t^n}{n!} \sum_{C \in \mathcal{L}_n^{ncm}(132)} 
y^{\cdes{C}} = ln\left( \frac{1}{1-y\int_{0}^t e^{(1-y)s - y\frac{s^2}{2}}ds}
\right).
\end{equation}
Then it follows that 
\begin{eqnarray}
NCM(t,x,y) &=& \sum_{n \geq 0} \frac{t^n}{n!} \sum_{\sg \in S_n} x^{cyc(\sg)}
y^{\cdes{\sg}}\\
&=& \sum_{n\geq 0} \frac{t^n}{n!} \sum_{k=1}^n x^k \sum_{\sg \in 
\mathcal{NCMS}_{n,k}(\tau)} y^{\cdes{\sg}} \nonumber \\
&=& 
e^{x\ ln\left( \frac{1}{1-y\int_{0}^t e^{(1-y)s - y\frac{s^2}{2}}ds}\right)} \nonumber \\
&=& \left( \frac{1}{1-y\int_{0}^t e^{(1-y)s - y\frac{s^2}{2}}ds}\right)^x. \nonumber
\end{eqnarray}

The outline of this paper is as follows. 
In Section 2, we determine the generating function 
$CA_\tau(t,x,y)$ and $NCM_\tau(t,x,y)$ for all $\tau \in S_3$ as well 
as compute $CA_\Upsilon(t,x,y)$ and $NCM_\Upsilon(t,x,y)$
for certain subsets $\Upsilon \subseteq S_3$. In 
section 3, we shall compute $NCM_\tau(t,x,y)$ for all 
$\tau = \tau_1 \ldots \tau_j \in S_j$ where $\tau_1 =$ and 
$\tau_j =2$ and for all $\tau = \tau_1 \ldots \tau_{j+p} \in S_{j+p}$ of the  
form $\tau = 1~2 \ldots J-1~\gamma~j$ where $j \geq 3$ and 
$\gamma$ is a permutation of $j+1, \ldots ,j+p$.  Finally, 
in Section 4, we shall briefly describe two other approaches 
to computing the generating function  $NCM_\tau(t,x,y)$.

\section{Patterns of length 3}

In this section, we study $CA_{\tau}(t,x,y)$ 
and $NCM_{\tau}(t,x,y)$ for $\tau \in S_3$.

First we consider  $CA_{\tau}(t,x)$ for $\tau \in S_3$. 
It follows from our remarks in 
the introduction that both  cycle avoidance Wilf equivalence and 
cycle matching Wilf equivalence are closed under the operation 
of reverse and complement. Thus  
\begin{enumerate}
\item $1~2~3 \sim_{ca} 3~2~1$ and $1~2~3 \sim_{cm} 3~2~1$ and   
\item $1~3~2 \sim_{ca} 2~3~1 \sim_{ca} 2~1~3 \sim_{ca} 3~1~2$ and  
$1~3~2 \sim_{cm} 2~3~1 \sim_{cm} 2~1~3 \sim_{cm} 3~1~2$.
\end{enumerate}
Now since cycle avoidance Wilf equivalence is closed 
under cycle rearrangements, it follows that  
$1~2~3 \sim_{ca} 2~3~1$ which means that all permutations 
of length three are cycle avoidance Wilf equivalent. 
Thus for all permutations $\tau$ of length three, we have 
\begin{equation*}
CA_{\tau}(t) = CA_{123}(t) = e^{e^t-1}.
\end{equation*}
But since 
$$CA_{\tau}(t) =e^{\sum_{m \geq 1} L_m^{ca}(\tau) \frac{t^m}{m!}}$$
for all $\tau \in S_3$, it must be the case that  
\begin{equation*}
\sum_{m \geq 1} L_m^{ca}(\tau)\frac{t^m}{m!} = e^t-1
\end{equation*}
for all $\tau \in S_3$ and, hence,  
$$
CA_{\tau}(t,x) = e^{x \sum_{m \geq 1} L_m^{ca}(\tau)\frac{t^m}{m!}} = 
e^{x(e^t-1)}
$$ 
for all $\tau \in S_3$. 
However it is not the case that the generating functions 
$CA_{\tau}(t,x,y)$ are equal for all $\tau \in S_3$. That is, suppose 
that  $\alpha$ is a cyclic rearrangement of $\beta$. 
Then it is easy to see that $\mathcal{L}^{ca}_m(\alpha) = 
\mathcal{L}^{ca}_m(\beta)$ for all $m \geq 1$ 
so that 
\begin{equation}
\sum_{C \in \mathcal{L}^{ca}_m(\alpha)}y^{\cdes{C}} = 
\sum_{C \in \mathcal{L}^{ca}_m(\beta)}y^{\cdes{C}}.
\end{equation}
But then it follows from Theorem \ref{Fundrefined} that we must 
have 
$CA_{\alpha}(t,x,y) = CA_{\beta}(t,x,y)$. It thus follows that 
from our results in the introduction that 
\begin{equation*}
CA_{123}(t,x,y)= CA_{312}(t,x,y)= CA_{231}(t,x,y)= e^{x\left(yt+ \frac{1}{y}(e^{yt}-1-yt)\right)}.
\end{equation*}

Next consider $\tau = 1~3~2$. It is easy to see that 
for $k \geq 3$, the only $k$-cycle which cycle avoids  
$\tau$ is the cycle $(1,2,\ldots,k)$. Thus  
\begin{equation*}
\sum_{C \in \mathcal{L}^{ca}_m(132)} y^{\cdes{C}} =y,
\end{equation*} 
for all $k \geq 1$. 
Hence 
\begin{equation*}
CA_{132}(t,x,y) = CA_{213}(t,x,y)= CA_{321}(t,x,y)= 
e^{x \left(\sum_{m \geq 1} \frac{yt^m}{m!}\right)} = 
e^{xy(e^{t}-1)}.
\end{equation*}

Next we shall consider the generating functions 
$NCM_\tau(t,x,y)$ for $\tau \in S_3$. We claim 
that is enough to compute $NCM_{123}(t,x,y)$ and 
$NCM_{132}(t,x,y)$.  That is, for any $j \geq 2$ and 
$\tau \in S_j$,  we can compute $NCM_{\tau^r}(t,x,y)$ and 
$NCM_{\tau^c}(t,x,y)$ from  $NCM_\tau(t,x,y)$.  
Note that it follows from Theorem \ref{Fundrefined} that 
\begin{equation}\label{trans:tau}
\sum_{n \geq 1} \frac{t^n}{n!} \sum_{C \in \mathcal{L}^{ncm}_n(tau)} 
y^{\cdes{C}} = ln\left(NCM_{\tau}(t,1,y)\right).
\end{equation}
Since $\sum_{C \in \mathcal{L}^{ncm}_1(123)} y^{\cdes{C}} =y$, it 
follows that 
\begin{equation}\label{transtau:2}
\sum_{n \geq 2} \frac{t^n}{n!} \sum_{C \in \mathcal{L}^{ncm}_n(\tau)} 
y^{\cdes{C}} = ln\left(NCM_{\tau}(t,1,y)\right)-yt.
\end{equation}
Given any $n$-cycle $C$ in $S_n$, let $C^{cr}$ denotes its cycle-reverse and 
$C^{cc}$ denotes its cycle-complement.  Then  
$C \in \mathcal{L}^{ncm}_n(\tau)$ if and only if 
$C^{cr} \in \mathcal{L}^{ncm}_n(\tau^r)$ and 
$C \in \mathcal{L}^{ncm}_n(\tau)$ if and only if 
$C^{cc} \in \mathcal{L}^{ncm}_n(\tau^c)$
Now if $n \geq 2$, then it is easy to see that  
$ n -\cdes{C} = \cdes{C^{cr}} = \cdes{C^{cc}}$. That is, 
each descent as we read once around the cycle $C$ becomes 
a rise as we read around the cycles of $C^{cr}$ and $C^{cc}$ and 
each rise as we read once around the cycle $C$ becomes 
a descent as we read around the cycles of $C^{cr}$ and $C^{cc}$.
Note, however, that if $C$ is a one-cycle, 
then $C^{cr} =C^{cc} = C$ and $\cdes{C} = \cdes{C^{cr}} = \cdes{C^{cc}}= 1$ so 
that it is not 
the case that $\cdes{C^{cr}} = \cdes{C^{cr}}=  1 -\cdes{C}$. 
Thus we have to treat 
the one-cycles separately. Thus we have that  
\begin{eqnarray*}
 \sum_{n \geq 2} \frac{t^n}{n!} \sum_{C \in \mathcal{L}^{ncm}_n(\tau)} 
y^{n-\cdes{C}} &=&  \sum_{n \geq 2} \frac{t^n}{n!} \sum_{C \in \mathcal{L}^{ncm}_n(\tau^r)} 
y^{\cdes{C}} \\
&=& \sum_{n \geq 2} \frac{t^n}{n!} \sum_{C \in \mathcal{L}^{ncm}_n(\tau^c) }
y^{\cdes{C}}.
\end{eqnarray*}
It follows that if $\tau \in S_j$ where $j \geq 2$ and 
\begin{equation}
G(t,y) = \sum_{n \geq 2} \frac{t^n}{n!} \sum_{C \in \mathcal{L}^{ncm}_n(\tau)} 
y^{\cdes{C}},
\end{equation}
then 
\begin{equation}
G(ty,y^{-1}) =  \sum_{n \geq 2} \frac{t^n}{n!} 
\sum_{C \in \mathcal{L}^{ncm}_n(\tau^r)} 
y^{\cdes{C}} =  \sum_{n \geq 2} \frac{t^n}{n!} 
\sum_{C \in \mathcal{L}^{ncm}_n(\tau^c)} 
y^{\cdes{C}}. 
\end{equation}
Thus by (\ref{transtau:2}), we have that 
\begin{eqnarray*}
ln\left( NCM_{\tau}(ty,1,y^{-1})\right) -t &=& \sum_{n \geq 2} \frac{t^n}{n!} 
\sum_{C \in \mathcal{L}^{ncm}_n(\tau^r)} 
y^{\cdes{C}} \\
&=& \sum_{n \geq 2} \frac{t^n}{n!} 
\sum_{C \in \mathcal{L}^{ncm}_n(\tau^c)} 
y^{\cdes{C}}
\end{eqnarray*}
so that 
\begin{eqnarray*}
ty-t + ln\left( NCM_{\tau}(ty,1,y^{-1})\right) 
&=& \sum_{n \geq 1} \frac{t^n}{n!} 
\sum_{C \in \mathcal{L}^{ncm}_n(\tau^r)} 
y^{\cdes{C}} \\
&=& \sum_{n \geq 1} \frac{t^n}{n!} 
\sum_{C \in \mathcal{L}^{ncm}_n(\tau^c)} 
y^{\cdes{C}}
\end{eqnarray*}
Then we can apply Theorem \ref{Fundrefined} to obtain the following 
result.

\begin{theorem}\label{tautotaurc} 
Let $\tau \in S_j$ where $j \geq 2$. Then 
\begin{equation}
NCM_{\tau^r}(t,x,y) = NCM_{\tau^c}(t,x,y) = 
e^{x(yt-t +ln(NCM_{\tau}(ty,1,y^{-1})))}.
\end{equation}
\end{theorem}

Next we shall show that we can find an explicit expression 
$NCM_{123}(t,x,y)$ using some results of Mendes and Remmel 
\cite{MenRem2}.  Suppose that we want to compute the generating function 
\begin{eqnarray}\label{gf:NCM}
NCM_{\tau}(t,x,y) &=&  \sum_{n \geq 0} 
\frac{t^n}{n!} \sum_{\sg \in \mathcal{NCMS}_n(\tau)} 
x^{cyc(\sg)} y^{\cdes{\sg}} \\
&=& e^{x\sum_{n\geq 1} \frac{t^n}{n!} \sum_{C \in 
\mathcal{L}^{ncm}_n(\tau)} y^{\cdes{\sg}}} \nonumber
\end{eqnarray}
in the case where $\tau$ starts with 1. 
Then by Corollary \ref{cor:starts1}, we know that  
\begin{equation}\label{gf:NCM2}
NCM_{\tau}(t,x,y) = NM_{\tau}(t,x,y)= \sum_{n \geq 0} 
\frac{t^n}{n!} \sum_{\sg \in \mathcal{NMS}_n(\tau)} 
x^{LtRMin(\sg)} y^{1+des(\sg)}.
\end{equation}
Now suppose that we can compute 
\begin{equation}\label{eq:1k0}
NM_{\tau}(t,1,y)= \sum_{n \geq 0} 
\frac{t^n}{n!} \sum_{\sg \in \mathcal{NMS}_n(\tau)} y^{1+des(\sg)}.
\end{equation}
Then we know that 
$$
NM_{\tau}(t,1,y) =   e^{\sum_{n\geq 1} \frac{t^n}{n!} \sum_{C \in 
\mathcal{L}^{ncm}_n(\tau)} y^{\cdes{\sg}}}
$$ 
so that 
$$ \sum_{n\geq 1} \frac{t^n}{n!} \sum_{C \in 
\mathcal{L}^{ncm}_n(\tau)} y^{\cdes{\sg}} = 
ln\left(NM_{\tau}(t,1,y)\right).
$$
But then it follows that 
\begin{equation}\label{eq:1k1}
NCM_{\tau}(t,x,y) = NM_{\tau}(t,x,y) = 
e^{x\ ln\left(NM_{\tau}(t,1,y)\right)}.
\end{equation}
Thus we need only compute (\ref{eq:1k0}). However, 
Mendes and Remmel \cite{MenRem2} proved the following theorem.

\begin{theorem} \label{theorem:j21} 
If $\tau = j \ldots 2~1$ where $j \geq 2$, then  
\begin{equation}\label{eq:MR}
\sum_{n \geq 0} \frac{t^n}{n!} \sum_{\sg \in \mathcal{NMS}_n(\tau)} 
y^{\des{\sg}}
= \left( \sum_{n \geq 0} \frac{t^n}{n!} \sum_{i \geq 0} (-1)^i  
\EuScript{R}_{n-1,i,j-1} y^i \right)^{-1}
\end{equation}
where $\EuScript{R}_{n,i,j}$ is the number of rearrangements of $i$ zeroes and
$n-i$ ones such that $j$ zeroes never appear consecutively.
\end{theorem}

Replacing $y$ by $1/y$ and then replacing $t$ by $yt$ in (\ref{eq:MR}) 
yields
\begin{equation}\label{eq:MR1}
\sum_{n \geq 0} \frac{t^n}{n!} \sum_{\sg \in \mathcal{NMS}_n(\tau)} 
y^{n-\des{\sg}}
= \left( \sum_{n \geq 0} \frac{t^n}{n!} \sum_{i \geq 0} (-1)^i  
\EuScript{R}_{n-1,i,j-1} y^{n-i} \right)^{-1}.
\end{equation}
It is easy to see that if $\sg \in S_n$ has no 
$j \ldots 2~1$-matches, then the reverse of $\sg$, $\sg^r$ has 
no $1~2 \ldots j$-matches and that $n -\des{\sg}$ equals 
$1+\des{\sg^r}$.  Thus it follows that if $\alpha = 1~2 \ldots j$, then 
\begin{equation}\label{eq:MR2}
\sum_{n \geq 0} \frac{t^n}{n!} \sum_{\sg \in \mathcal{NMS}_n(\alpha)} 
y^{1+\des{\sg}}
= \left( \sum_{n \geq 0} \frac{t^n}{n!} \sum_{i \geq 0} (-1)^i  
\EuScript{R}_{n-1,i,j-1} y^{n-i} \right)^{-1}.
\end{equation}
Thus we have the following theorem. 

\begin{theorem}\label{thm:12j}
For $j \geq 2$ and $\tau = 12 \ldots j$, 
\begin{eqnarray}\label{gf:NCM3}
NCM_{\tau}(t,x,y) &=&  \sum_{n \geq 0} 
\frac{t^n}{n!} \sum_{\sg \in \mathcal{NCMS}_n(\tau)} 
x^{cyc(\sg)} y^{\cdes{\sg}} \\
&=& e^{x\ ln\left( \frac{1}{\sum_{n \geq 0} 
\frac{t^n}{n!} \sum_{i \geq 0} (-1)^i  
\EuScript{R}_{n-1,i,j-1} y^{n-i}}\right)} \nonumber \\
&=& \left( \frac{1}{\sum_{n \geq 0} 
\frac{t^n}{n!} \sum_{i \geq 0} (-1)^i  
\EuScript{R}_{n-1,i,j-1} y^{n-i}}\right)^x. \nonumber 
\end{eqnarray}
where $\EuScript{R}_{n,i,j}$ is the number of rearrangements of $i$ zeroes and
$n-i$ ones such that $j$ zeroes never appear consecutively.
\end{theorem}

Now if $\tau = 123$, then we can obtain a more explicit 
formula for $NCM_{\tau}(t,x,y)$ using the following 
observations of Mendes and Remmel \cite{MenRem2}. That is, suppose that 
we start with a word $w =w_1 \ldots w_n$ which is a sequence in 
$\{0,1\}^*$ with no two consecutive zeros. Then we can 
uniquely factor $w$ by cutting the word before each 0. For example, 
if $w = 11110110111010101110$ then we would factor $w$ as 
$$1111|011|0111|01|01|0111|0.$$
It is easy to see that each such word $w$ is of the form 
$$\{1\}^*\{01^i:i \geq 1\}^*(\epsilon + 0)$$
where $\epsilon$ is the empty word. Thus if $U$ is the set 
of a words in $\{0,1\}^*$ with no two consecutive zeros and we 
weight each word in $w \in U$ by $WT(w) = y^{1(w)}z^{0(w)}t^{|w|}$ where 
$1(w)$ is the number of 1's in $w$, $0(w)$ is the number of 0's in $w$, 
and $|w|$ is the length of $w$, then it follows that 
\begin{eqnarray}\label{ogfU}
U(t,y,z) &=& \sum_{w \in U} WT(w) \nonumber \\
&=& \frac{1}{1-yt}\frac{1}{1- \sum_{n \geq 2} 
y^{n-1}zt^n}(1+zt) \nonumber \\
&=& \frac{1+zt}{(1-yt - yzt^2)}.
\end{eqnarray}
But then it is easy to see that 
\begin{equation}
\sum_{i \geq 0} (-1)^i  
\EuScript{R}_{n-1,i,j-1} y^{n-i} = yU(t,y,-1)|_{t^{n-1}}.
\end{equation}

Thus we have the following corollary  of Theorem \ref{thm:12j}. 

\begin{corollary}\label{cor:123}
\begin{eqnarray}\label{gf:NCM123}
NCM_{123}(t,x,y) &=&  \sum_{n \geq 0} 
\frac{t^n}{n!} \sum_{\sg \in \mathcal{NCMS}_n(123)} 
x^{cyc(\sg)} y^{\cdes{\sg}} \\
&=& e^{x\ ln\left( \frac{1}{\sum_{n \geq 0} 
\frac{t^n}{n!} \left(\frac{y(1-t)}{1-yt+yt^2}|_{t^{n-1}}\right)}\right)}
\nonumber \\
&=& \left( \frac{1}{\sum_{n \geq 0} 
\frac{t^n}{n!} \left(\frac{y(1-t)}{1-yt+yt^2}|_{t^{n-1}}\right)}\right)^x
\end{eqnarray}
\end{corollary}

One can use our generating functions for 
$NCM_{123}(t,x,y)$ to compute the initial values of 
 $L_n^{ncm}(123)$ and  $NCM_n(123)$.

\begin{center}
\begin{tabular}{|c|c|c|} 
\hline 
$n$ & $L_n^{ncm}(123)$ & $NCM_n(123)$ \\
\hline
1& 1 & 1\\
\hline
2 & 1 & 2\\
\hline
3 & 1 & 5 \\
\hline
4 & 3 & 17\\
\hline
5 & 9 & 70\\
\hline
6 & 39 & 349 \\
\hline
7 & 189 & 2017\\
\hline  
8  & 1107 & 13358\\
\hline
9  & 7281 & 99377\\
\hline
10  & 54351 & 822041\\
\hline
\end{tabular}
\end{center}
If one looks in the OEIS, one will see that both sequences occur. 
That is, the sequence of $L_n^{ncm}(123)$ is sequence 
A080635 and counts the number of permutations on $n$ letters 
without double falls and without an initial fall. The sequence 
for $NCM_n(123)$ counts the number of permutations in $S_n$ which 
have no 123-matches as expected.

Next we will compute $NCM_{132}(t,x,y)$. In this case, we 
will directly compute 
\begin{equation}
L_{132}(t,y) = 
\sum_{m \geq 1} \frac{t^m}{m!} \sum_{C \in \mathcal{L}^{ncm}_m} y^{\cdes{C}}.
\end{equation}
We start with a general observation. 
Suppose $\tau = \tau_1 \ldots \tau_j \in S_j$ where $\tau_1 =1$. 
We can write any $n$-cycle $C$ in the form $C= (\alpha_1, \ldots, \alpha_n)$ 
where $\alpha_1 =1$.  It is easy to see that the only  
 cycle $\tau$-match in $C$ that can involve $\alpha_1 =1$ is 
$\alpha_1~\alpha_2 \ldots \alpha_j$.  This means that the only 
possible cycle $\tau$-matches in $C$ must be of the form 
$\alpha_i~\alpha_{i+1} \ldots \alpha_{i+j-1}$ where $i \leq n-j+1$. Thus 
the problem of finding $n$-cycles with no cycle $\tau$-matches 
is equivalent to the problem of finding permutations 
$\sg = \sg_1 \ldots \sg_n$ where $\sg_1 =1$ and $\sg$ has 
no $\tau$-matches. Let $S^1_n$ denote the set of all 
permutations $\sg = \sg_1 \ldots \sg_n \in S_n$ such that 
$\sg_1 =1$ and let $S^1_{n,\tau} = S^1_n \cap \mathcal{NMS}_n(\tau)$ 
be the set of permutations of $S^1_n$ with no $\tau$-matches. Then 
 \begin{equation}\label{AtoL1}
A_{n,\tau}(y) =   \sum_{\sg \in S^1_{n,\tau}} y^{1+\des{\sg}} =  \sum_{C \in \mathcal{L}^{ncm}_n} y^{\cdes{C}}.
\end{equation}
It turns out 
that in many cases we can find recurrences for $A_{n,\tau}(y)$ 
by classifying the permutations $\sg = \sg_1 \ldots \sg_n \in S_n$  
such that $\sg_1 =1$ according the position of 2 in $\sg$. Let $\mathcal{E}_{n,k,\tau}$ denote the set of permutations 
$\sg = \sg_1 \ldots \sg_n \in S^1_n(\tau)$ 
such that $\sg_k =2$. 

Now fix $\tau = 1~3~2$ and let 
$A_m(y) = A_{m,\tau}(y)$ and 
$\mathcal{E}_{n,k}= \mathcal{E}_{n,k,\tau}$. Our goal is 
compute $A(t,y) = \sum_{m \geq 1} \frac{A_{m}(y)t^m}{m!}$. 
Now $A_1(y) = A_2(y) =y$ since the permutation 
$1$ has no $\tau$-matches and $1+\des{1} =1$ and the permutation $1~2$ has no $\tau$-matches and  $1+\des{12} =1$. There 
are two permutations in $S_3$ that start with 1, namely, 
$1~2~3$ and $1~3~2$ and only  $1~2~3$ has no $\tau$-matches so that 
$A_3(y) = y$ since $1+\des{123} =1$.  
Now suppose that $n \geq 4$. 
Every permutation 
in $\mathcal{E}_{n,2}$ is of the form $1~2~\sg_3 \ldots \sg_n$. 
Clearly, the only $\tau$-matches must be of the form 
$\sg_i~\sg_{i+1}~\sg_{i+2}$ where $i \geq 2$ so that 
$\mathcal{E}_{n,2}$ contributes $A_{n-1}(y)$ to $A_n(y)$. 
Every permutation 
in $\mathcal{E}_{n,3}$ is of the form $1~\sg_2~2 \ldots \sg_n$ where 
$\sg_2 \geq 3$. Thus all such permutations have a $\tau$-match so 
that $\mathcal{E}_{n,3}$ contributes nothing to $A_n(y)$. 
For $4 \leq k \leq n$, 
the elements of the $\mathcal{E}_{n,k}$ are of the form 
$$1~\sg_2 \ldots \sg_{k-1}~2~\sg_{k+1} \ldots \sg_n.$$
In such a case, the only way that 2 can be part of $\tau$-match 
is if the $\tau$-match is $2~\sg_{k+1}~\sg_{k+2}$. It follows that an element 
of $\mathcal{E}_{n,k}$ contributes to $A_n(y)$ only if 
there is no $\tau$-match in $\sg_1 \ldots \sg_{k-1}$ and there 
is no $\tau$-match in $2~\sg_{k+1} \ldots \sg_{n}$. Note that 
since $\sg_{k-1} 2$ is adescent pair,  
$$1+\des{1~\sg_2 \ldots \sg_{k-1}~2~\sg_{k+1} \ldots \sg_n} = 
1+\des{1~\sg_2 \ldots \sg_{k-1}}+1+\des{2~\sg_{k+1} \ldots \sg_n}.
$$
Hence the contribution 
of $\mathcal{E}_{n,k}$ to $A_n(y)$ is just 
$\binom{n-2}{k-2}A_{k-1}(y)A_{n-k+1}(y)$ since there are $\binom{n-2}{k-2}$ 
to choose the elements which make up $\sg_2, \ldots, \sg_{k-1}$.  
Thus for $n \geq 4$, 
\begin{equation}\label{132rec1}
A_n(y) =  A_{n-1}(y) + \sum_{k=4}^n \binom{n-2}{k-2} A_{k-1}(y)A_{n-k+1}(y).
\end{equation}
Dividing  both sides of (\ref{132rec1}) by $(n-2)!$, we obtain that 
for all $n \geq 4$, 
\begin{equation}\label{132rec2}
\frac{A_n(y)}{(n-2)!} = \frac{A_{n-1}(y)}{(n-2)!}+
\sum_{k=2}^{n-2} \frac{A_{k+1}(y)}{k!} \frac{A_{n-k-1}(y)}{(n-2-k)!}.
\end{equation}  
If we multiply both sides of (\ref{132rec2}) by $t^{n-2}$ and 
sum, we obtain the differential equation 
\begin{equation*}
\frac{\partial^2A(t,y)}{\partial t^2}-y -yt = 
\frac{\partial A(t,y)}{\partial t} -y -yt +
\left(\frac{\partial A(t,y)}{\partial t}-y -yt\right)
\frac{\partial A(t,y)}{\partial t}
\end{equation*} 
so that $A(t,y)$ satisfies the second order partial differential equation 
\begin{equation}\label{132difeq}
\frac{\partial^2A(t,y)}{\partial t^2} = 
\frac{\partial A(t,y)}{\partial t}(1-y-yt) + 
\left(\frac{\partial A(t,y)}{\partial t}\right)
\end{equation} 
with initial conditions $A_0(y) = 0$ and $A_1(y) =y$. One can check 
that the solution to (\ref{132difeq}) is 
\begin{equation}
A(t,y) = ln\left(\frac{1}{1-y\int_0^t e^{(1-y)s-ys^2/2}ds}\right)
\end{equation}
Hence 
\begin{eqnarray}
L_{132}(t,y) &=& 
\sum_{m \geq 1} \frac{t^m}{m!} \sum_{C \in \mathcal{L}^{ncm}_m(132)} y^{\cdes{C}}.
\nonumber \\
&=& ln\left(\frac{1}{1-y\int_0^t e^{(1-y)s-ys^2/2}ds}\right)
\end{eqnarray}

Thus we have the following theorem.
\begin{theorem}\label{thm:132}
\begin{eqnarray}
NCM_{132}(t,x,y) &=& 
e^{x \ ln\left(\frac{1}{1-y\int_0^t e^{(1-y)s-ys^2/2}ds}\right)} \nonumber \\
&=& \frac{1}{\left(1-y\int_0^t e^{(1-y)s-ys^2/2}ds\right)^x}.
\end{eqnarray}
\end{theorem}

We note that specialization 
$$NCM_{132}(t,1,1) = \frac{1}{1-\int_0^t e^{-s^2/2}ds}
$$ 
has been proved by Elizalde and Noy \cite{Eliz}.

One can use our generating functions for 
$NCM_{132}(t,x,y)$ to compute the initial values of 
 $L_n^{ncm}(132)$ and  $NCM_n(132)$.

\begin{center}
\begin{tabular}{|c|c|c|}
\hline 
$n$ & $L_n^{ncm}(132)$ & $NCM_n(132)$ \\
\hline
1& 1& 1\\
\hline
2 & 1 & 2\\
\hline
3 & 1 & 5 \\
\hline
4 & 2 & 16\\
\hline
5 & 7 & 63\\
\hline
6 & 28 & 296 \\
\hline
7 & 131 & 1623\\
\hline  
8  & 720 & 10176\\
\hline
9  & 4513 & 71793\\
\hline
10  & 31824 & 562848\\
\hline
\end{tabular}
\end{center}

If one looks in the OEIS, then both the sequences for 
$L_n^{ncm}(132)$ and $NCM_n(132)$ occur.  The sequence for 
$L_n^{ncm}(132)$ is sequence A052319 which 
counts the number of increasing rooted 
trimmed trees with $n$ nodes. Here an increasing tree 
is a tree labeled with $1, \ldots,n$ where the numbers 
increase as you move away from the root. A tree with 
a forbidden limb of length $k$ is a tree where the 
path from any leaf inward hits a branching node or another leaf within 
$k$ steps. A trimmed tree is a tree with a forbidden limb of length 2.  
The sequence 
for $NCM_n(132)$ is the number of permutations that have no 
132-matches as expected. 

We end this section with some results on $CA_{\Upsilon}(t,x,y)$ and 
$NCM_{\Upsilon}(t,x,y)$ where $\Upsilon \subseteq S_3$. 
For certain $\Upsilon$'s, this problem is uninteresting. 
For example, if $\Upsilon$ contains both 
$1~2~3$ and $1~3~2$, then any $k$-cycle 
$C =(\sg_1,\sg_2, \ldots, \sg_k)$ where $\sg_1 =1$ and $k \geq 3$ will have 
a cycle $\Upsilon$-match since $\sg_1~\sg_2~\sg_3$ must be either a 
cycle $1~2~3$-match or a cycle $1~3~2$-match. Thus in this 
case $\mathcal{L}_1^{ca}(\Upsilon) = \mathcal{L}_1^{ncm}(\Upsilon) = 
\{(1)\}$, $\mathcal{L}_2^{ca}(\Upsilon) = \mathcal{L}_2^{ncm}(\Upsilon) = \{(1,2)\}$,  and 
$\mathcal{L}_k^{ca}(\Upsilon) = \mathcal{L}_k^{ncm}(\Upsilon) = \emptyset$ 
for $k \geq 3$. 
 It then follows 
from Theorem \ref{Fundrefined} that 
\begin{equation*}
CA_{\Upsilon}(t,x,y) = NCM_{\Upsilon}(t,x,y) = e^{x\left(yt+\frac{yt^2}{2}\right)}
\end{equation*}

A more interesting case is when 
$\Upsilon =\{123,321\}$. First observe that since 
any cycle contains a cycle occurrence of $1~3~2$ if and only if 
it contains a cycle occurrence of $3~2~1$, then it is the case that  
any $k$-cycle $C$ where $k \geq 3$ must have  a cycle occurrence 
of either $1~2~3$ or $3~2~1$. Thus 
\begin{equation*}
CA_{\Upsilon}(t,x,y) = e^{x\left(yt+\frac{yt^2}{2}\right)}
\end{equation*}

Let 
$C =(\sg_1, \ldots, \sg_n)$ be an $n$-cycle such that 
$\sg_1 =1$.  If $n \geq 3$, then we must have 
$\sg_2 > \sg_3$ since otherwise there will be a 
cycle $1~2~3$-match.  But then we must have 
$\sg_3 < \sg_4$ since otherwise there would 
be cycle $3~2~1$-match. Continuing on in this way, we 
see that $\sg_2 \ldots \sg_n$ must be an alternating 
permutation. That is, we must have 
$$\sg_2 > \sg_3 < \sg_4 > \sg_5 < \sg_6 > \sg_7 \cdots.$$ 
However, this means if $n = 2k+1 \geq 3$, then there are no 
$n$ cycles which have no cycle $\Upsilon$-matches since 
since we are forced to have $\sg_{2k}> \sg_{2k+1} > \sg_1$ which 
is a cycle $3~2~1$-match.  If $n =2k$ and 
$\sg_2 \ldots \sg_n$ is alternating, then $C$ will have 
no cycle $\Upsilon$-matches. For such $\sg$ it is easy 
to see that $1+\des{\sg} = k$. Thus in this case, 
$L_{2k+1}^{ncm}(\Upsilon) = 0$ for $k \geq 1$ and 
$L_{2k}^{ncm}(\Upsilon)$ is just the number of odd alternating 
permuations of length $2k-1$ for $k \geq 1$. 

If we let $Alt_n$ denote 
the number of Alternating permutations of length $n$, then 
Andr\'{e} \cite{Andre1,Andre2} proved that 
\begin{equation}
\sum_{n\geq 0} Alt_{2n+1} \frac{t^{2n+1}}{(2n+1)!} = \frac{sin(t)}{cos(t)}.
\end{equation}
Thus 
\begin{eqnarray*}
\sum_{n\geq 1} L_{2n}^{ncm}(\Upsilon) \frac{t^{2n}}{(2n)!} 
&=&\sum_{n\geq 1} Alt_{2n-1} \frac{t^{2n}}{(2n)!}  \\
&=&\int_0^t \frac{sin(z)}{cos(z)}dz = -ln(|cos(t)|.
\end{eqnarray*}
Hence, 
\begin{eqnarray*}
\sum_{n\geq 1} \frac{t^{2n}}{(2n)!} 
\sum_{C \in \mathcal{L}_{2n}^{ncm}(\Upsilon)} y^{\cdes{C}} 
&=&\sum_{n\geq 1} y^n L_{2n}^{ncm}(\Upsilon) \frac{t^{2n}}{(2n)!} \\
&=& -ln(|cos(t\sqrt{y})|.
\end{eqnarray*}
and 
\begin{equation}
\sum_{n\geq 1} \frac{t^{n}}{(n)!} 
\sum_{C \in \mathcal{L}_{n}^{ncm}(\Upsilon)} y^{\cdes{C}} = 
ty-ln(|cos(t\sqrt{y})|.
\end{equation}
It follows that 
\begin{equation}
NCM_{\Upsilon}(t,x,y) = e^{x(ty-ln(|cos(t\sqrt{y})|)} = 
\frac{e^{xyt}}{cos(t\sqrt{y})^x} =e^{xyt} sec(t\sqrt{y})^x.
\end{equation}

Thus we have proved the following theorem.
\begin{theorem}
\begin{equation}
\sum_{n\geq 1} \frac{t^{n}}{(n)!} 
\sum_{C \in \mathcal{L}_{n}^{ncm}(\{123,321\}} y^{\cdes{C}} = 
ty-ln(|cos(t\sqrt{y})|
\end{equation}
and 
\begin{equation}
NCM_{\{123,321\}}(t,x,y) =e^{xyt} (sec(t\sqrt{y}))^x.
\end{equation}
\end{theorem}

\section{General results}

In this section, we shall describe how we can compute 
$NCM_{\tau}(t,x,y)$ for certain general classes of permutations  
$\tau$. We start by considering permutations 
 $\tau = \tau_1 \ldots \tau_j$ where 
$\tau_1 =1$ and $\tau_j =2$. In that case, we have the following 
theorem. 

\begin{theorem}\label{thm:1-2} 
Let $\tau = \tau_1 \ldots \tau_j \in S_j$ where 
$j \geq 3$ and $\tau_1 =1$ and $\tau_j =2$.
Then 
\begin{equation}\label{1alpha2}
NCM_{\tau}(t,x,y) = \frac{1}{(1 - \int_0^t e^{(y-1)s- \frac{y^{\des{\tau}}s^{j-1}}{(j-1)!}}ds)^x}
\end{equation}
\end{theorem}
\begin{proof}
Note that in the special case where $j=3$, the only permutation satisfying 
the hypothesis of the theorem is 
$\tau = 1~3~2$. Thus in this special case, the result follows 
from Theorem \ref{thm:132}. Thus assume that we fix a 
$\tau = \tau_1 \ldots \tau_j \in S_j$ where 
$\tau_1 =1$ and $\tau_j =2$ and $j \geq 4$.

Our first goal is to compute 
\begin{equation}
A(t,y) = \sum_{n \geq  1} A_n(y) \frac{t^n}{n!}
\end{equation}
where $A_n(y) = \sum_{\sg \in S^1_{n,\tau}} y^{\des{\sg}+1}$. 
Now it is easy to see that 
$A_n(y) = \sum_{\sg \in S^1_n} y^{\des{\sg}+1}$ for $1 \leq n \leq j-1$. 
Thus 
\begin{eqnarray*}
A(t,y) &=& y t + y \frac{t^2}{2} + (y+y^2) \frac{t^3}{3!} + \cdots \\
\frac{\partial A(t,y)}{\partial t}  &=& y  + y t + (y+y^2) \frac{t^2}{2!} + \cdots \ \mbox{and}\\
\frac{\partial^2 A(t,y)}{\partial t^2}  &=& y + (y+y^2)t + \cdots .
\end{eqnarray*}
For $n \geq j$, we shall prove a recursive formula for 
$A_n(y)$. We 
consider three cases for $\sg = \sg_1 \ldots \sg_n \in S^1_{n,\tau}$ depending 
on the position of 2 in $\sg$. \\
\ \\
{\bf Case 1.} $\sg_2 =2$.\\
In this case because $j \geq 4$, the only possible $\tau$-matches 
in $\sg$ must occur in $\sg_2 \ldots \sg_n$.  Since 
$\des{\sg}+1 = \des{\sg_2 \ldots \sg_n}+1$, it follows 
that the contribution of the permutations in this case to 
$A_n(y)$ is just $A_{n-1}(y)$. \\
\ \\
{\bf Case 2.}  $\sg_k =2$ where $k \notin \{2,j\}$.\\
In this case, we have 
$\binom{n-2}{k-2}$ ways to choose the elements $D_k$ that will constitute 
$\sg_2 \ldots \sg_{k-1}$. Once we have chosen $D_k$, we have to 
consider the ways in which we can arange the elements of 
$D_k$ to form  
$\sg_2 \ldots \sg_{k_1}$ and the ways that we  
can arrange $[n]-(D_k \cup \{1,2\})$ to form $\sg_{k+1} \ldots \sg_{n}$ 
so that 
\begin{equation}\label{2=k}
\sg = 1~\sg_2 \ldots \sg_{k-1}~2~\sg_{k+1} \ldots \sg_{n}
\end{equation}
has no $\tau$-matches.  However  it is easy to see 
that since $k \notin \{2,j\}$ that the only $\tau$-matches 
for $\sg$ of the form (\ref{2=k}) can occur in either 
entirely in $1~\sg_2 \ldots \sg_{k-1}$ or entirely in 
$2~\sg_{k+1} \ldots \sg_n$.  Moreover it is the case 
that 
$$\des{\sg} +1 = \des{1~\sg_2 \ldots \sg_{k-1}}+1 + \des{2~\sg_{k+1} \ldots \sg_n}+1$$
since $\sg_{k-1} > 2$.  Thus the contribution to 
$A_n(y)$ of the permutations in this case is 
$$\binom{n-2}{k-2}A_{k-1}(y) A_{n-k+1}(y).$$ \\
\ \\
{\bf Case 3.}  $\sg_j =2$.\\
In this case, we have 
$\binom{n-2}{j-2}$ ways to choose the elements $D_j$ that will constitute 
$\sg_2 \ldots \sg_{j-1}$. Once we have chosen $D_j$, we have to 
consider the ways in which we can arange the elements of 
$D_j$ to form  
$\sg_2 \ldots \sg_{j_1}$  and we  
can arrange $[n]-(D_j \cup \{1,2\})$ to form $\sg_{j+1} \ldots \sg_{n}\sg_{k+1} \ldots \sg_{n}$ 
so that 
\begin{equation}\label{2=j}
\sg = 1~\sg_2 \ldots \sg_{j-1}~2~\sg_{j+1} \ldots \sg_{n}
\end{equation}
has no $\tau$-matches.  Unlike Case 2, it is not 
enough just to ensure that $ 1~\sg_2 \ldots \sg_{j-1}$ and 
$2~\sg_{j+1} \ldots \sg_{n}$ have no $\tau$-matches. That is, 
we must also ensure that $\red{\sg_2 \ldots \sg_{j-1}} \neq \red{\tau_2 \ldots 
\tau_{j-1}}$ since otherwise $1~\sg_2 \ldots \sg_{j-1}~2$ would be 
$\tau$-match.  Note that in such a situation 
$\des{1~\sg_2 \ldots \sg_{j-1}}+1 = \des{\tau}$.  Thus the contributions 
to $A_n(y)$ 
of the permutations in this case is 
$$\binom{n-2}{j-2} (A_{j-1}(y) - y^{\des{\tau}})A_{n-j+1}(y).$$

It follows that for $ n \geq j$, 
\begin{equation}
 A_n(y) =  A_{n-1}(y) + \sum_{k=3}^n \binom{n-2}{k-2} A_{k-1}(y) A_{n-k+1}(y)  - \binom{n-2}{j-2} y^{\des{\tau}} A_{n-j+1}(y) 
\end{equation} 
or, equivalently,  
\begin{equation}\label{1alpha2rec}
 \frac{A_n(y)}{(n-2)!}  = \frac{A_{n-1}(y)}{(n-2)!} + 
\left( \sum_{k=3}^n \frac{A_{k-1}(y)}{(k-2)!} 
\frac{A_{n-k+1}(y)}{(n-k)!} \right) 
- \frac{y^{\des{\tau}}}{(j-2)!} \frac{A_{n-j+1}(y)}{(n-j)!}. 
\end{equation}
Now for any formal power series 
$f(t) = \sum_{n\geq 1} f_nt^n$, we let 
$f(t)|_{t^{\leq j}}$ denote $f_0 +f_1t + \cdots + f_jt^j$. We then 
then multiple both sides of (\ref{1alpha2rec}) by $t^{n-2}$ and sum 
and we will get the differential equation 
\begin{eqnarray*} \label{1alpha2difeq1} 
&&\frac{\partial^2 A(t,y)}{\partial t^2} - \left(\frac{\partial^2 A(t,y)}{\partial t^2}|_{t^{\leq j-3}}\right) \\
&&= \frac{\partial A(t,y)}{\partial t} - \left(\frac{\partial A(t,y)}{\partial t}|_{t^{\leq j-3}}\right) + \nonumber \\
&& \ \ \ \left(\frac{\partial A(t,y)}{\partial t}-y\right)\frac{\partial A(t,y)}{\partial t} - \left(\left(\frac{\partial A(t,y)}{\partial t}-y\right)\frac{\partial A(t,y)}{\partial t}|_{t^{\leq j-3}}\right) - \nonumber \\
&&\ \ \ \frac{y^{\des{\tau}}}{(j-2)!} \frac{\partial A(t,y)}{\partial t}.
\end{eqnarray*}
Thus 
\begin{eqnarray*}
\frac{\partial^2 A(t,y)}{\partial t^2} &=& (1-y- y^{\des{\tau}})
\frac{\partial A(t,y)}{\partial t} + \left(\frac{\partial A(t,y)}{\partial t}\right)^2 +\\
&& \left(\frac{\partial^2 A(t,y)}{\partial t^2}|_{t^{\leq j-3}}\right)
- \left(\frac{\partial A(t,y)}{\partial t}|_{t^{\leq j-3}}\right) - 
\left(\left(\frac{\partial A(t,y)}{\partial t}-y\right)\frac{\partial A(t,y)}{\partial t}|_{t^{\leq j-3}}\right).
\end{eqnarray*}
We claim that 
\begin{equation*} 
0=\left(\frac{\partial^2 A(t,y)}{\partial t^2}|_{t^{\leq j-3}}\right)
- \left(\frac{\partial A(t,y)}{\partial t}|_{t^{\leq j-3}}\right) - 
\left(\left(\frac{\partial A(t,y)}{\partial t}-y\right)\frac{\partial A(t,y)}{\partial t}|_{t^{\leq j-3}}\right)
\end{equation*}
or, equivalently, that 
\begin{equation}\label{special1}
\frac{\partial^2 A(t,y)}{\partial t^2}|_{t^{\leq j-3}}
= \left( \frac{\partial A(t,y)}{\partial t} + 
\left(\frac{\partial A(t,y)}{\partial t}-y\right)\frac{\partial A(t,y)}{\partial t}\right)|_{t^{\leq j-3}}.
\end{equation}
If we take the coefficient of $t^s$ where $0 \leq s \leq t^{j-3}$ 
on both sides 
of (\ref{special1}), then we must show that 
\begin{eqnarray*}
\frac{A_{s+2}(y)}{s!} &=&  \frac{A_{s+1}(y)}{s!} + \sum_{k=1}^s 
\frac{A_{k+1}(y)}{k!}\frac{A_{s-k+1}(y)}{(s-k)!} \\
&=&  \frac{A_{s+1}(y)}{s!} + \sum_{k=3}^{s+2} 
\frac{A_{k-1}(y)}{(k-2)!}\frac{A_{s+2-(k-1)}(y)}{(s+2-k)!}.
\end{eqnarray*}
Thus if we multiply both sides by $s!$, we see that we must 
show that for $0 \leq s \leq j-3$, 
\begin{equation}\label{special2}
A_{s+2}(y) = A_{s+1}(y) + \sum_{k=3}^{s+2} \binom{s+2}{k-2} 
A_{k-1}(y) A_{s+2 -(k-1)}(y).
\end{equation}
However this follows from our analysis of Cases 1, 2, and 3 above for 
the recursion of $A_{s+2}(y)$. 
That is, since $s+2 \leq j-1$, Case 2 does not apply so 
that we only get the contributions from Cases 1 and 3 which 
is exactly (\ref{special2}). 

Thus we have shown that $A(y,t)$ satisfies the partial differential 
equation where 
\begin{equation}
\frac{\partial^2 A(t,y)}{\partial t^2} = (1-y- y^{\des{\tau}})
\frac{\partial A(t,y)}{\partial t} + \left(\frac{\partial A(t,y)}{\partial t}\right)^2
\end{equation}
 with intial conditions that $A(y,0) =0$, 
$A(y,t)|_t = y$, and $A(y,t)|_{\frac{t^2}{2!}} = y$.  
It is then easy to check that the solution to this PDE is 
\begin{equation} 
A(y,t) = ln\left(\frac{1}{1 - \int_0^t e^{(1-y)s + y^{\des{\tau}}\frac{s^{j-1}}{(j-1)!}} ds}\right).
\end{equation}
Thus   
\begin{eqnarray}\label{special3}
A(y,t) &=& \sum_{n\geq 1} \frac{t^n}{n!} \sum_{C \in \mathcal{L}^{ncm}_n(\tau)} y^{\cdes{C}} \nonumber \\
&=& 
ln\left(\frac{1}{1 - \int_0^t e^{(1-y)s + y^{\des{\tau}}\frac{s^{j-1}}{(j-1)!}} ds}\right).
\end{eqnarray}
But then we know by Theorem \ref{Fundrefined}, that 
\begin{eqnarray*}
NCM_{\tau}(t,x,y) &=& e^{x \sum_{n \geq 1} \frac{t^n}{n!}\sum_{C \in \mathcal{L}^{ncm}_n(\tau)} y^{\cdes{C}}} \\
&=& e^{x  ln \left(\frac{1}{1 - \int_0^t e^{(1-y)s + y^{\des{\tau}}\frac{s^{j-1}}{(j-1)!}} ds}\right)}\\
&=&\left(\frac{1}{1 - \int_0^t e^{(1-y)s + y^{\des{\tau}}\frac{s^{j-1}}{(j-1)!}} ds}\right)^x
\end{eqnarray*}
which is what we wanted to prove. 
\end{proof}

We end this section by showing how one can compute $NCM_\tau(t,x,y)$ where 
 $\tau \in S_m$ is of the form $\tau=1~2~\dots~(j-1)~\gamma~j$ where 
$\gamma$ is a permutation of the elements $j+1, \ldots, m$ where 
$m \geq j+1$.  We let $p = m-j$ so that $\red{\gamma} \in S_p$.  
We shall assume that $j \geq 3$ since we have already dealt with permutations that start with 1 
and end with 2. 

Using our previous theorems as a guide, we shall assume that 
$NCM_\tau(t,x,y)$ is of the form 
\begin{equation}\label{final1}
NCM_\tau(t,x,y) = e^{x\sum_{n \geq 1} \frac{t^n}{n!} \sum_{C \in \mathcal{L}^{ncm}_n(\tau)} y^{\cdes{C}}}\nonumber = \frac{1}{(U_\tau(t,y))^x}
\end{equation}
where 
\begin{equation}\label{Udef1}
U_\tau(t,y) = \sum_{n \geq 0} U_{n,\tau} \frac{t^n}{n!}.
\end{equation}
We have been unable to find a closed form for 
$U_\tau(t,y)$. However, we can show that the coefficients of 
$U_{n,\tau}(y)$ satisfy a simple recursion. 
That is, we shall prove the following.

\begin{theorem}\label{thm:1-j-1-sg-j} Suppose that $\tau = 1~2~\ldots j-1~\gamma~j$ where 
$\gamma$ is a permutation of $j+1, \ldots , j+p$ and $j \geq 3$. 
Then 
$$NCM_\tau(t,x,y) = \frac{1}{(U_\tau(t,y))^x}$$ 
where 
\begin{equation}
U_\tau(t,y) = \sum_{n \geq 0} U_{n,\tau}(y) \frac{t^n}{n!}
\end{equation}
and 
\begin{equation}\label{Urec}
U_{n+j,\tau}(y) = (1-y)U_{n+j-1,\tau}(y) - y^{\des{\tau}}\binom{n}{p} 
U_{n-p+1,\tau}(y).
\end{equation}
\end{theorem}
\begin{proof}

Taking the natural logarithm of both sides (\ref{final1}) 
and using (\ref{AtoL1}), we see  
\begin{equation}\label{Udef2}
-ln(U_\tau(t,y))=  \sum_{n \geq 1} \frac{t^n}{n!} \sum_{\sg \in \mathcal{L}^{ncm}_n(\tau)} y^{\des{\sg}+1} 
 \sum_{n \geq 1} \sum_{\sg \in S^1_{n,\tau}} y^{\des{\sg}+1}.
\end{equation}

Before proceeding, we need to establish some notation. Fix $\tau$ of 
the form $1~2~\ldots j-1 \gamma j$ where $j \geq 3$.   
For any $\sg \in S_n^1$, we let $\tau$-$\mbox{imch}(\sg)$ be the 
indicator function that the initial segment of size $m$ in 
$\sg$ is a $\tau$-match. Thus 
$\tau$-$\mbox{imch}(\sg) = 1$ if $\red{\sg_1 \ldots \sg_m} =\tau$ and 
we let $\tau$-$\mbox{imch}(\sg) = 0$ otherwise. For $i=1, \ldots, j-1$, 
we let $\tau^{(i)} = \red{i~i+1~ \ldots j-1~\gamma~j}$. 
Our first goal is to compute  
\begin{equation}
A(t,y)= \sum_{n \ge 1}A_n(y)\frac{t^n}{n!}
\end{equation}
where 
$$A_n(y) = \sum_{\sigma \in S^1_{n,\tau}}y^{1+\text{des}(\sigma)}.$$
For $i =2, \ldots, k-1$, we shall also need the following functions 
\begin{equation}
B_i(t,y)= 1+\sum_{n \ge 1}B_{i,n}(y)\frac{t^n}{n!} 
\end{equation}
where 
$$B_{i,n}(y) = \sum_{\substack{\sigma \in S^1_n \\ \tau\text{-mch}(\sigma)=0 \\  \tau^{(2)} \text{imch}(\sigma)=0 \\ \tau^{(3)} \text{imch}(\sigma)=0\\ \vdots \\ \tau^{(i)} \text{imch}(\sigma)=0}}y^{1+\text{des}(\sigma)}.$$
Thus $B_{i,n}(y)$ is the sum of $y^{1+\text(des)(\sg)}$ over 
all permutation $\sg$ in $S_n^1$ such that $\sg$ has no $\tau$-matches 
and $\sg$ does not start with a $\tau^{(j)}$-match for $j =2, \ldots, i$.

First we develop recursions for $A_n(y)$ for $n \geq 2$. Let 
$\mathcal{E}_{n,k,\tau}$ denote the set of all 
$\sg = \sg_1 \ldots \sg_n \in S^1_{n,\tau}$ such 
that $\sg_k =2$. We then consider 
two cases for $\sg \in S^1_{n,\tau}$ depending on 
which $\mathcal{E}_{n,k,\tau}$ contains   $\sg$.\\
\ \\
{\bf Case 1.} $\sg \in \mathcal{E}_{n,2,\tau}$.\\
 Thus $\sg = 1~2~\sg_3 \ldots \sg_n$.  To ensure 
that $\sg$ has no $\tau$-matches, we must ensure that 
there are no $\tau$ matches in $2~\sg_3 \ldots \sg_n$ and 
that $\sg$ does not start with a $\tau$-match which is 
equivalent to ensuring that $2~\sg_3 \ldots \sg_n$ does 
not start with $\tau^{(2)}$-match. Thus in this case, the 
permtuations of $ \mathcal{E}_{n,2,\tau}$ contribute 
$B_{2,n-1}(y)$ to $A_n(y)$. \\
\ \\
{\bf Case 2}  $\sg \in  \mathcal{E}_{n,k,\tau}$ where $3 \leq k \leq n$.\\
In this case, it is easy to see that the only possible 
$\tau$-matches must occur in $\sg_k \ldots \sg_n$ or 
in $\sg_1 \ldots \sg_{k-1}$.  Thus we have 
$\binom{n-2}{k-2}$ ways to choose that elements that will constitute 
$\sg_2 \ldots \sg_{k-1}$ and $A_{k-1}(1)$ ways to order them so 
that there are no $\tau$-matches in $\sg_1 \ldots \sg_{k-1}$.
Once we have picked $\sg_2 \ldots \sg_{k-1}$, there are 
$A_{n-k+1}(1)$ ways to order the remaining elements so 
that there are no $\tau$-matches in $\sg_k \ldots \sg_n$.
Having picked $\sg$, we have that 
$$\text{des}(\sg)+1 = \text{des}(\sg_1 \ldots \sg_{k-1})+1 
+\text{des}(\sg_k \ldots \sg_{n})+1$$ 
since $\sg_{k-1} > 2$. Hence in this case,  
the permutations in $\mathcal{E}_{n,k,\tau}$ where will contribute 
$ \binom{n-2}{k-2} A_{k-1}(y) A_{n-k+1}(y)$ 
elements to $A_n(y)$.\\
\ \\
It follows that for $n \geq 2$, 
\begin{equation}\label{AB2}
A_n(y) = B_{2,n-1}(y) + \sum_{k=3}^n \binom{n-2}{k-2} A_{k-1}(y)A_{n-k+1}(y).
\end{equation}

We can develop similar recursions for $B_{2,n}(y)$ for $n \geq 2$. 
However we have to consider the cases $j=3$ and 
$j>3$ separately.

First consider, the case where $j=3$.  Note in 
this case $\tau^{(2)} = \red{2~\gamma~3} = 1~\alpha~2$ where 
$\alpha$ is a permutation of $3,\ldots,p+2$ such that  
$\red{\alpha} = \red{\gamma}$. 
We then consider 
three cases for $\sg \in S^1_{n,\tau}$ depending on 
which $\mathcal{E}_{n,k,\tau}$ contains   $\sg$.\\
\ \\
{\bf Case 1.} $\sg \in  E_{n,2,\tau}$.\\
 Thus $\sg = 1~2~\sg_3 \ldots \sg_n$.  To ensure 
that $\sg$ has no $\tau$-matches, we must ensure that 
there are no $\tau$ matches in $2~\sg_3 \ldots \sg_n$ and 
that $\sg$ does not start with a $\tau$-match which is 
equivalent to ensuring that $2~\sg_3 \ldots \sg_n$ does 
not start with $\tau^{(2)}$-match. 
It might seem that to ensure that $\sg$ does not start with 
a $\tau^{(2)}$-match that we must ensure that $2~\sg_3 \ldots \sg_n$ does start with $\tau^{(3)}$-match. However, in this case 
$\tau^{(3)} = \red{\gamma~3}$ does not start with 1 so that 
is automatically true that $2~\sg_3 \ldots \sg_n$ does start with $\tau^{(3)}$-match. Thus the 
permutations in $\mathcal{E}_{n,2,\tau}$ contribute 
$B_{2,n-1}(y)$ to $B_{2,n}(y)$. \\
\ \\
{\bf Case 2.} $\sg \in \mathcal{E}_{n,p+2,\tau}$.\\
In this case, it is easy to see that the only possible 
$\tau$-matches must occur in $\sg_{p+1} \ldots \sg_n$ or 
in $\sg_1 \ldots \sg_{p}$.
Now we have  
$\binom{n-2}{p}$ ways to choose that elements that will constitute 
$\sg_2 \ldots \sg_{p+1}$. We can order these elements 
in any way that we want except that we cannot 
have  $\red{\sg_2 \ldots \sg_{p+1}} = \red{\gamma}$ since otherwise 
$\sg$ would start with at $\tau^{(2)}$ match. 
Note that 
$B_{2,p+1}(y) = \sum_{\beta \in S^1_{p+1}} y^{\des{\beta}+1}$ since 
no permutation of length $p+1$ can contain a $\tau$-match or start 
with $\tau^{(2)}$-match.  Since 
$$\des{1~\sg_2 \ldots \sg_{p+1}}+1 + \des{2~\sg_{p+2} \ldots \sg_n} +1 = 
\des{\sg}$$
and $\des{1~\gamma}+1 = \des{\tau}$, 
the permutations in $\mathcal{E}_{n,p+2,\tau}$ will contribute \\
$ \binom{n-2}{p} (B_{2,p+1}(y) -y^{\des{\tau}})A_{n-p-1}(y)$ 
to $B_{2,n}(y)$.\\
\ \\
{\bf Case 3.}  $\sg \in \mathcal{E}_{n,k,\tau}$ where $3 \leq k \leq n$ and 
$k \notin \{2,p+2\}$.\\
In this case, it is easy to see that the only possible 
$\tau$-matches must occur in $\sg_k \ldots \sg_n$ or 
in $\sg_1 \ldots \sg_{k-1}$.  Thus we have 
$\binom{n-2}{k-2}$ ways to choose that elements that will constitute 
$\sg_2 \ldots \sg_{k-1}$ and $B_{2,k-1}(1)$ ways to order them so 
that there are no $\tau$-matches in $\sg_1 \ldots \sg_{k-1}$ and 
$\sg_1 \ldots \sg_{k-1}$ does 
not start with a $\tau^{(2)}$ match and $A_{n-k+1}(1)$ ways 
to order $\sg_k \ldots \sg_n$ that it contains no $\tau$-match.
It follows that  
the permutations in $\mathcal{E}_{n,k,\tau}$ will contribute 
$ \binom{n-2}{k-2} B_{2,k-1}(y) A_{n-k+1}(y)$ 
to $B_{2,n}(y)$.\\
\ \\
Thus if $n \geq p+2$,  we have the recursion 
\begin{equation}\label{B2j=3a}
B_{2,n}(y) = B_{2,n-1}(y) + \left(\sum_{k=3}^n \binom{n-2}{p-2} 
B_{2,k-1}(y)A_{n-k+1}(y) \right) - \binom{n-2}{p}y^{\des{\tau}}A_{n-p-1}(y).
\end{equation}
For $2 \leq n \leq p+1$, Case 2 does not apply so that we have 
the recursion 
\begin{equation}\label{B2j=3b}
B_{2,n}(y) = B_{2,n-1}(y) + \left(\sum_{k=3}^n \binom{n-2}{p-2} 
B_{2,k-1}(y)A_{n-k+1}(y) \right).
\end{equation}

Before considering the case where $j >3$, we shall 
show how we can derive a recursion (\ref{Urec}) for the 
$U_{n,\tau}(y)$s in this case. We have shown that for all $n \geq 2$, 
\begin{eqnarray*}
A_n(y) &=& B_{2,n-1}(y) + 
\sum_{k=3}^n \binom{n-2}{k-2} A_{k-1}(y)A_{n-k+1}(y) \ \mbox{and} \\
B_{2,n}(y) &=& B_{2,n-1}(y) + \left(\sum_{k=3}^n \binom{n-2}{p-2} 
B_{2,k-1}(y)A_{n-k+1}(y) \right) - \\
&& \ \ \ \chi(n \geq p+2)y^{\des{\tau}}\binom{n-2}{p}A_{n-p-1}(y)
\end{eqnarray*}
where for any statement $A$, we let $\chi(A)$ equal 1 if $A$ is true and 
equal 0 if $A$ is false. Thus we have that for all $n \geq 2$, 
\begin{eqnarray*}
\frac{A_n(y)}{(n-2)!} &=& \frac{B_{2,n-1}(y)}{(n-2)!} + 
\sum_{k=3}^n \frac{A_{k-1}(y)}{(k-2)!}\frac{A_{n-k+1}(y)}{(n-k)!} 
\ \mbox{and} \\
\frac{B_{2,n}(y)}{(n-2)!} &=& \frac{B_{2,n-1}(y)}{(n-2)!} + \left(\sum_{k=3}^n  \frac{B_{2,k-1}(y)}{(k-2)!}\frac{A_{n-k+1}(y)}{(n-k)!} \right) - \chi(n \geq p+2)\frac{y^{\des{\tau}}}{p!}\frac{A_{n-p-1}(y)}{(n-p)!}.
\end{eqnarray*}
Multiplying by $t^{n-2}$ and summing, we obtain the following differential 
equations when we think of $A=A(t,y)$ and $B_2 = B_2(t,y)$ as just functions 
of $t$:
\begin{eqnarray*}
A^{\prime \prime} &=& B_2^\prime +(A^\prime -y)A^{\prime} \ \mbox{and} \\
B_2^{\prime \prime} &=& B_2^\prime +(B_2^\prime -y)A^{\prime} - 
\frac{y^{\des{\tau}}t^p}{p!}A^\prime.
\end{eqnarray*}
Now if $U = U(t,y) = U_\tau(t,y)$, then $A = -ln(U)$. Thus 
\begin{eqnarray}
A^\prime &=& \frac{-U^\prime}{U} \ \mbox{and} \\
A^{\prime \prime} &=& \frac{-U^{\prime \prime}}{U} +\left( \frac{U^\prime}{U}\right)^2.
\end{eqnarray} 
Making these substititions in our first differential equation and 
solving for $B_2^\prime$, we see that 
\begin{equation}\label{eq:B2prime}
B_2^\prime = -\frac{U^{\prime \prime} +y U^\prime}{U}.
\end{equation}
Thus 
\begin{equation}\label{eq:B2prime2}
B_2^{\prime \prime} = -\frac{U^{\prime \prime \prime} +y U^{\prime \prime}}{U} +\frac{(U^{\prime \prime} +y U^{\prime})U^\prime}{U^2}.
\end{equation}
Substituting these expressions into our second differential equation and 
simplifying, we obtain the following differential equation for 
$U$, 
\begin{equation}\label{eq:Uj=3}
U^{\prime \prime \prime} = (1-y) U^{\prime \prime}- \frac{y^{\des{\tau}}t^p}
{p!}U^\prime.
\end{equation}
Taking the coefficient of $\frac{t^n}{n!}$ on both sides of 
(\ref{eq:Uj=3}), we set that 
\begin{equation}
U_{n+3,\tau}(y) = (1-y)U_{n+2}(y) - \binom{n}{p}y^{\des{\tau}}U_{n-p+1}(y).
\end{equation}
in the case where $\tau = 1 2 \gamma 3$ and $\gamma$ is permutation 
of $4,\ldots, 3+p$.

Now consider the recursion for $B_{2,n}(y)$ where $j > 3$. 
We then consider  
two cases for $\sg \in S^1_{n,\tau}$ depending on 
which set $\mathcal{E}_{n,k,\tau}$ contains $\sg$.\\
\ \\
{\bf Case 1.} $\sg \in \mathcal{E}_{n,2,\tau}$.\\
 Thus $\sg = 1~2~\sg_3 \ldots \sg_n$.  To ensure 
that $\sg$ has no $\tau$-matches, we must ensure that 
there are no $\tau$ matches in $2~\sg_3 \ldots \sg_n$ and 
that $\sg$ does not start with a $\tau$-match which is 
equivalent to ensuring that $2~\sg_3 \ldots \sg_n$ does 
not start with $\tau^{(2)}$-match. However in this case, 
we must also ensure that $\sg$ does not start with 
at $\tau^{(2)}$ which means that $2~\sg_3 \ldots \sg_n$ must  
not start with $\tau^{(3)}$-match.
Thus in this case, the 
$\sg \in \mathcal{E}_{n,2,\tau}$ contribute 
$B_{3,n-1}(y)$ to $B_{2,n}(y)$. \\
\ \\
{\bf Case 2}  $\sg \in \mathcal{E}_{n,k,\tau}$ where $3 \leq k \leq n$.\\
In this case, it is easy to see that the only possible 
$\tau$-matches must occur in $\sg_k \ldots \sg_n$ or 
in $\sg_1 \ldots \sg_{k-1}$.  Thus we have 
$\binom{n-2}{k-2}$ ways to choose that elements that will constitute 
$\sg_2 \ldots \sg_{k-1}$ and $B_{2,k-1}(1)$ ways to order them so 
that there are no $\tau$-matches in $\sg_1 \ldots \sg_{k-1}$ and 
$\sg_1 \ldots \sg_{k-1}$ does 
not start with a $\tau^{(2)}$ match and there are $A_{n-k+1}(1)$ ways 
to order $\sg_{k} \ldots \sg_{n}$ so that there is no $\tau$-match. 
It follows that  
the permutations in $\mathcal{E}_{n,k}$ will contribute 
$ \binom{n-2}{k-2} B_{2,k-1}(y) A_{n-k+1}(y)$ 
elements to $B_{2,n}(y)$.\\
\ \\
It follows that if $j \geq 3$, then for $n \geq 2$, 
\begin{equation}\label{B2B3}
B_{2,n}(y) = B_{3,n-1}(y) + \sum_{k=3}^n \binom{n-2}{k-2} B_{2,k-1}(y)A_{n-k+1}(y).
\end{equation}

One can repeat this type of argument to show that in general, 
for $2 \leq i \leq j-2$  
\begin{equation}\label{BiBi+1}
B_{i,n}(y) = B_{i+1,n-1}(y) + \sum_{k=3}^n \binom{n-2}{k-2} B_{i,k-1}(y)A_{n-k+1}(y).
\end{equation}

The recursion for $B_{j-1,n}(y)$ is similar to the recursion for 
$B_{2,n}(y)$ when $j =3$. That is, $\tau^{(j-1)} 
= \red{j-1~\gamma~j} = 1~\alpha~2$, where $\alpha$ is a permutation of 
$3, \ldots, p+2$ and $\red{\gamma} =\red{\alpha}$. 
Then we have to consider 
three cases depending on 
which set $\mathcal{E}_{n,k,\tau}$ contains $\sg$.\\
\ \\
{\bf Case 1.} $\sg \in \mathcal{E}_{n,2,\tau}$.\\
 Thus $\sg = 1~2~\sg_3 \ldots \sg_n$.  To ensure 
that $\sg$ has no $\tau$-matches and does 
not start with $\tau^{(i)}$-match  for $i =2, \ldots, j-1$, we 
clearly have to ensure that 
$2~\sg_3 \ldots \sg_n$ has no $\tau$-matches and does 
not start with $\tau^{(i)}$-match  for $i =2, \ldots, j-1$. 
However, we do not have to worry about 
$2~\sg_3 \ldots \sg_n$ starting with 
$\tau^{(j)} = \red{\sg~j}$ since $\tau^{(j)}$ does not 
start with its least element. 
Thus in this case, the permutations in 
$\mathcal{E}_{n,2,\tau}$ contribute 
$B_{j-1,n-1}(y)$ to $B_{j-1,n}(y)$. \\
\ \\
{\bf Case 2.} $\sg \in \mathcal{E}_{n,p+2,\tau}$
In this case, it is easy to see that the only possible 
$\tau$-matches must occur in $\sg_{p+1} \ldots \sg_n$ or 
in $\sg_1 \ldots \sg_{p}$.
Now we have  
$\binom{n-2}{p}$ ways to choose that elements that will constitute 
$\sg_2 \ldots \sg_{p+1}$. We can order these elements 
in any way that we want except that we cannot 
have  $\red{\sg_2 \ldots \sg_{p+1}} = \red{\gamma}$ since otherwise 
$\sg$ would start with at $\tau^{(j-1)}$ match. 
Note that 
$B_{j-1,p+1}(y) = \sum_{\beta \in S^1_{p+1}} y^{\des{\beta}+1}$ since 
no permutation of length $p+1$ can contain a $\tau$-match or start 
with $\tau^{(i)}$-match for $i=2,\ldots j-1$.  Thus since  
$$\des{1~\sg_2 \ldots \sg_{p+1}}+1 + \des{2~\sg_{p+2} \ldots \sg_n} +1 = 
\des{\sg}$$
and $\des{1~\gamma}+1 = \des{\tau}$, the permutations in 
$\mathcal{E}_{n,p+2,\tau}$ will contribute \\
$ \binom{n-2}{p} (B_{j-1,p+1}(y) -y^{\des{\tau}})A_{n-p+1}(y)$ 
to $B_{j-1,n}(y)$.\\
\ \\
{\bf Case 3.}  $\sg \in \mathcal{E}_{n,k,\tau}$ where $3 \leq k \leq n$ and 
$k \notin \{2,p+2\}$.\\
In this case, it is easy to see that the only possible 
$\tau$-matches must occur in $\sg_k \ldots \sg_n$ or 
in $\sg_1 \ldots \sg_{k-1}$.  Thus we have 
$\binom{n-2}{k-2}$ ways to choose that elements that will constitute 
$\sg_2 \ldots \sg_{k-1}$ and $B_{j-1,k-1}(1)$ ways to order them so 
that there are no $\tau$-matches in $\sg_1 \ldots \sg_{k-1}$ and 
$\sg_1 \ldots \sg_{k-1}$ does 
not start with a $\tau^{(i)}$-match for $i=2,\ldots, j-1$ and 
there are $A_{n-k+1}(1)$ ways 
to order $\sg_{k} \ldots \sg_{n}$ so that there is no $\tau$-match.  
Thus  
the permutations in $\mathcal{E}_{n,k,\tau}$ will contribute 
$ \binom{n-2}{k-2} B_{j-1,k-1}(y) A_{n-k+1}(y)$ 
to $B_{j-1,n}(y)$.\\
\ \\
It follows that for $n \geq 2$, 
\begin{eqnarray}\label{Bj-1Bj-1}
B_{j-1,n}(y) &=& B_{j-1,n-1}(y) + \sum_{k=3}^n \binom{n-2}{k-2} B_{j-1,k-1}(y)A_{n-k+1}(y) - \\
&&\chi(n \geq p+2)\binom{n-2}{p}y^{\des{\tau}}A_{n-p-1}(y). \nonumber 
\end{eqnarray}

Thus for all $n \geq 2$, we have proved that in general 

\begin{align*}
A_n(y)&=B_{2,n-1}(y)+\sum_{k=3}^n {n-2 \choose k-2}A_{k-1}(y)A_{n-k+1}(y)\\
B_{2,n}(y)&=B_{3,n-1}(y)+\sum_{k=3}^n {n-2 \choose k-2}B_{2,k-1}(y)A_{n-k+1}(y)\\
B_{3,n}(y)&=B_{4,n-1}(y)+\sum_{k=3}^n {n-2 \choose k-2}B_{3,k-1}(y)A_{n-k+1}(y)\\
&\vdots\\
B_{j-2,n}(y)&=B_{j-1,n-1}(y)+\sum_{k=3}^n {n-2 \choose k-2}B_{j-2,k-1}(y)A_{n-k+1}(y)\\
B_{j-1,n}(y)&=B_{j-1,n-1}(y)+\left(\sum_{k=3}^n {n-2 \choose k-2}B_{j-1,k-1}(y)A_{n-k+1}(y)\right) - \\
& \ \ \ \ \ \chi(n \geq p+2) 
{n-2 \choose p} y^{\des{\tau}} A_{n-p-1}(y)
\end{align*}

As in the case for $j=3$, 
if we multiply everything by $\frac{t^n}{n!}$ and then sum over $n$ we get the following system of differential equations where we think of 
$A(t,y)$ and $B_i(t,y)$ for $i=2, \ldots, j-1$ as functions of $t$.

\begin{align*}
&(D_1) \ \ \ \ A''=B'_2 +A'^2-yA'\\
&(D_2) \ \ \ \ B_2''=B'_3 +B_2'A'-yA'\\
&(D_3) \ \ \ \ B_3''=B'_4 +B_3'A'-yA'\\
&\ \  \vdots\\
&(D_{j-2})\ B_{j-2}''=B'_{j-1} +B_{j-2}'A'-yA'\\
&(D_{j-1})\ B_{j-1}''=B'_{j-1} +B_{j-1}'A'-yA'-
\frac{t^{p}}{(p)!}y^{\des{\tau}}A'\\
\end{align*}

As in the case $j=3$, we let $A(t,y)=-\log(U(t,y))$ so 
that $A'=\frac{-U'}{U}$ and $A''=\frac{-U''}{U}+\frac{U'^2}{U^2}$. 
Thus under this substitution, the first differential equations becomes  
\begin{equation*}
\frac{-U''}{U}+\frac{U'^2}{U^2}=B_2'+\frac{U'^2}{U^2}+y\frac{U'}{U}
\end{equation*}
so that 
\begin{equation}\label{D2normal}
B_2' =\frac{-U''-yU'}{U}
\end{equation}

In fact, we have the following lemma.
\begin{lemma}  For $2 \le i \le j-1$, 
\begin{equation}\label{Dinormal}
B_i'=\frac{-U^{(i)}-y\sum_{k=1}^{i-1}U^{(k)}}{U}.
\end{equation}
\end{lemma}
\begin{proof}
We proceed by induction on $i$. We have already shown that 
(\ref{Dinormal}) in the case where $i=2$. Now suppose that 
\begin{equation}\label{Di}
B_i'=\frac{-U^{(i)}-y\sum_{k=1}^{i-1}U^{(k)}}{U}.
\end{equation}
Then we must show that  
\begin{equation}\label{Di+1}
B_{i+1}'=\frac{-U^{(i+1)}-y\sum_{k=1}^{i}U^{(k)}}{U}.
\end{equation}
Taking the derivative of both sides of (\ref{Di}) with respect to $t$, 
we see that 
$$B_i''=\frac{-U^{(i+1)}-y\sum_{k=2}^{i}U^{(k)}}{U}+(\frac{U^{(i)}+y\sum_k^{i-1}U^{(k)}}{U})(\frac{U'}{U}).$$
Pluggin our expression for $B_i''$ and $B_i'$ into the differential 
equation $(D_i)$, we see that 
\begin{eqnarray*}
&&\frac{-U^{(i+1)}-y\sum_{k=2}^{i}U^{(k)}}{U}+(\frac{U^{(i)}+y\sum_k^{i-1}U^{(k)}}{U})(\frac{U'}{U})\\
&&=B'_{i+1}+(\frac{-U^{(i)}-y\sum_{k=1}^{i-1}U^{(k)}}{U})(\frac{-U'}{U})-y(\frac{-U'}{U}).
\end{eqnarray*}
Solving for $B_{i+1}'$ we see that 
\begin{equation*}
B_{i+1}' =\frac{-U^{(i+1)}-y\sum_{k=1}^{i}U^{(k)}}{U}.
\end{equation*}
\end{proof}
 
By the Lemma, we know that  
\begin{equation*}
B_{j-1}' =\frac{-U^{(j-1)}-y\sum_{k=1}^{j-2}U^{(k)}}{U},
\end{equation*}
and, hence, 
\begin{equation*}
B_{j-1}'' =\frac{-U^{(j)}-y\sum_{k=2}^{j-1}U^{(k)}}{U}+(\frac{U^{(j-1)}+y\sum_k^{j-2}U^{(k)}}{U})(\frac{U'}{U}).
\end{equation*}
Thus plugging these expressions into the differential equation 
$(D_{j-1})$, we obtain that 
\begin{eqnarray*}
&&\frac{-U^{(j)}-y\sum_{k=2}^{j-1}U^{(k)}}{U}+(\frac{U^{(j-1)}+y\sum_k^{j-2}U^{(k)}}{U})(\frac{U'}{U})\\
&&= \frac{-U^{(j-1)}-y\sum_{k=1}^{j-2}U^{(k)}}{U}+\\
&&(\frac{-U^{(j-1)}-y\sum_{k=1}^{j-2}U^{(k)}}{U})(\frac{-U'}{U})-y(\frac{-U'}{U})-\frac{t^{p}}{p!}y^{\text{des}(\tau)}(\frac{-U'}{U}).
\end{eqnarray*}
Simplifying this expression yields that 
\begin{equation}\label{FDj-1}
U^{(j)}= (1-y)U^{(j-1)}-\frac{t^{p}}{p!}y^{\text{des}(\tau)}U'.
\end{equation}
Then taking the coefficient of $\frac{t^n}{n!}$ on both side 
of (\ref{FDj-1}) gives that 
$$U_{n+j}=(1-y)U_{n+j-1}+y^{\des{\tau}}{n \choose p}U_{n-p+1}$$
which is what we wanted to prove. 
\end{proof}

We end this section with an example of the use of 
Theorem \ref{thm:1-j-1-sg-j}. Let $\tau = 1243$ and 
\begin{equation*}
A_{n,\tau}(t,y) = \sum_{n \geq 1} A_{n,\tau}(y) \frac{t^n}{n!} = \sum_{n \geq 1} \frac{t^n}{n!} \sum_{\sg \in S^1_{n,\tau}} y^{\des{\sg}+1} = \sum_{n \geq 1} \frac{t^n}{n!} \sum_{C \in \mathcal{L}^{ncm}_n(\tau)}.
\end{equation*}
Thus it is easy to check that $A_{1,\tau}(y)  =y$, 
$A_{2,\tau}(y)  =y$, 
$A_{3,\tau}(y)  =y+y^2$, and  
$A_{4,\tau}(y)  =y+3y^2+y^3$.
Now  
$$U_{\tau}(t,y) = \sum_{n\geq 0} U_{n,\tau}(y) = e^{-A_\tau(t,y)}$$
so that one can use Mathematica to compute that $U_{0,\tau}(y)  =1$, 
$U_{1,\tau}(y)  = -y$, \\
$U_{2,\tau}(y)  =-y +y^2 y$,\ $U_{3,\tau}(y)  =-y+2 y^2-y^3$, and  
$U_{4,\tau}(y)  = -y+4 y^2-3 y^3+y^4$.\\
\ \\
By Theorem  \ref{thm:1-j-1-sg-j}, we know that we have 
the recursion that 
$$U_{n+3,\tau}(y) = (1-y) U_{n+2,\tau}(y) - y U_{n,\tau}(y).$$
Thus we can use this recursion to compute that \\
$U_{5,\tau}(y)  = -y+6 y^2-8 y^3+4 y^4-y^5$,\\
$U_{6,\tau}(y)  = -y+8 y^2-16 y^3+13 y^4-5 y^5+y^6$,\\
$U_{7,\tau}(y)  = -y+10 y^2-28 y^3+32 y^4-19 y^5+6 y^6-y^7$, and \\
$U_{8,\tau}(y)  = -y+12 y^2-44 y^3+68 y^4-55 y^5+26 y^6-7 y^7+y^8$.\\
\ \\
But then we know that 
$NCM_\tau(t,x,y) = \frac{1}{(U_\tau(t,y))^x}$. Thus one can use 
Mathematica to show that 
$$ 
NCM_\tau(t,x,y) = \sum_{n\geq 0} S^{ncm}_{n,\tau}(x,y) \frac{t^n}{n!},
$$
where  
$S^{ncm}_{0,\tau}(x,y) =1$,\    
$S^{ncm}_{1,\tau}(x,y) = xy$, \  
$S^{ncm}_{2,\tau}(x,y) = x y+x^2 y^2$,\\
\ \\
$S^{ncm}_{3,\tau}(x,y) = x y+x y^2+3 x^2 y^2+x^3 y^3$,\\
\ \\
$S^{ncm}_{4,\tau}(x,y) =x y+3 x y^2+7
x^2 y^2+x y^3+4 x^2 y^3+6 x^3 y^3+x^4 y^4$, \\
\ \\
$S^{ncm}_{5,\tau}(x,y) = x y+9 x y^2+15 x^2 y^2+8 x y^3+25 x^2 y^3+25 x^3 y^3+x y^4+5 x^2 y^4+10
x^3 y^4+10 x^4 y^4+x^5 y^5$,\\
\ \\
$S^{ncm}_{6,\tau}(x,y) = x y+23 x y^2+31 x^2 y^2+45 x y^3+119 x^2 y^3+90 x^3 y^3+20 x y^4+73 x^2 y^4+105 x^3 y^4+65
x^4 y^4+x y^5+6 x^2 y^5+15 x^3 y^5+20 x^4 y^5+15 x^5 y^5+x^6 y^6$, \\
\ \\
$S^{ncm}_{7,\tau}(x,y) =x y+53 x y^2+63 x^2 y^2+217 x y^3+490 x^2 y^3+301
x^3 y^3+192 x y^4+623 x^2 y^4+749 x^3 y^4+350 x^4 y^4+47 x y^5+196 x^2 y^5+343 x^3 y^5+315 x^4 y^5+140 x^5 y^5+x y^6+7 x^2 y^6+21 x^3 y^6+35 x^4
y^6+35 x^5 y^6+21 x^6 y^6+x^7 y^7$, and \\
\ \\
$S^{ncm}_{8,\tau}(x,y) =x y+115 x y^2+127 x^2 y^2+916 x y^3+1838 x^2 y^3+966 x^3 y^3+1500 x y^4+4333 x^2
y^4+4466 x^3 y^4+1701 x^4 y^4+765 x y^5+2810 x^2 y^5+4214 x^3 y^5+3164 x^4 y^5+1050 x^5 y^5+105 x y^6+495 x^2 y^6+1008 x^3 y^6+1148 x^4 y^6+770 x^5
y^6+266 x^6 y^6+x y^7+8 x^2 y^7+28 x^3 y^7+56 x^4 y^7+70 x^5 y^7+56 x^6 y^7+28 x^7 y^7+x^8 y^8$.

\section{Conclusions}
As mentioned in the introduction, we know of 
two other ways to compute $NCM_\tau(t,x,y)$ and 
$NCM_\Upsilon(t,y)$ for various $\tau$'s and $\Upsilon$'s.

Our second approach again uses the function 
$U_\tau(t,x,y)$ as defined in the previous section where 
$$NCM_\tau(t,x,y) = \sum_{n\geq 0} ncmS_{n,\tau}(x,y) \frac{t^n}{n!} = 
\frac{1}{(U_\tau(t,y))^x}.$$
It follows that 
\begin{equation}\label{conclusion1}
U_\tau(t,y) = \frac{1}{NCM_\tau(t,1,y)} = 
\frac{1}{\sum_{n\geq 0} ncmS_{n,\tau}(x,y) \frac{t^n}{n!}}.
\end{equation}
Remmel and his coathors \cite{Bec4,Lan2,MenRem1,MenRem2,MenRem3,MRR,Remriehl,wag} developed a method called 
the homomorphism method to show that many generating 
functions involving permutation statistics can be 
applied to simple symmetric function identities such as 
\begin{equation}\label{conclusion2}
H(t) = 1/E(-t)
\end{equation}
where 
$$H(t) = \sum_{n\geq 0} h_n t^n = \prod_{i\geq 1} \frac{1}{1-x_it}$$
is the generating function of the homogeneous symmetric functions  
$h_n$ in infinitely many variables $x_1,x_2, \ldots $ and 
$$E(t) = \sum_{n\geq 0} e_n t^n = \prod_{i\geq 1} 1+x_it $$
is the generating function of the elementary symmetric functions  
$e_n$ in infinitely many variables $x_1,x_2, \ldots $.
Now if we define a homomorphism $\theta$ and the ring of 
symmetric function so that 
$$\theta(e_n) = \frac{(-1)^n}{n!} ncmS_{n,\tau}(1,y),$$
then 
$$\theta(E(-t)) = \frac{1}{\sum_{n\geq 0} ncmS_{n,\tau}(1,y) \frac{t^n}{n!}}.$$
Thus $\theta(H(t))$ should equal $U_\tau(t,y)$.  One can 
then use the combinatorial methods associated with 
the homomorphism method to develop recursions for 
the coefficient of $U_\tau(t,y)$ much like we did 
in Theorem \ref{thm:1-j-1-sg-j}. For example, 
we can show that 
$$U_{n,1324}(y) = (1-y)U_{n-1,1324}(y)+ \sum_{k=2}^{\lfloor n/2 \rfloor} 
(-y)^{k-1} C_{k-1} U_{n-2k+1,1324}(y)
$$
where $C_k$ is $k$-th Catalan number and 
$$U_{n,1423}(y) = (1-y)U_{n-1,1423}(y)+ \sum_{k=2}^{\lfloor n/2 \rfloor} 
(-y)^{k-1} \binom{n-k-1}{k-1} U_{n-2k+1,1423}(y).
$$

The second author has 
 developed a third way to approach 
the problem of computing $NCM_\Upsilon(t)$ which is 
completely different from the other two approaches. That method involves 
defining a certain bijection between the set of derangements and certain 
fillings of brick tabloids. That bijection allows one  
to compute generating functions the number derangements 
that have no cycle $\Upsilon$-matches by applying an appropriate  ring 
homomorphism defined on the ring of symmetric functions $\Lambda$ in 
infinitely many variables $x_1,x_2, \ldots$ to certain simple symmetric 
function identities as described above. One can then multiply the resulting generating function  
by $e^t$ to obtain generating for $ncmS_n(\tau)$.   
This approach is generally much more complicated 
than the first two approaches. However, it allows us to compute 
$NCM_{\Upsilon}(t)$ for a number of sets of permutations 
$\Upsilon$  which seem beyond the either the techniques employed in this paper 
or the second approach described above. For example, one 
can show that 
$$NCM_\Upsilon(t) = \frac{e^t}{1- \sum_{n\geq 3} \frac{2(-t)^n}{n!}}$$  
where $\Upsilon$ is the set of permutations that contain 
$1234$ and all permutations $\sg =\sg_1 \sg_2 \sg_3 \sg_4 \sg_5$ such 
such that $\sg_1 < \sg_2 > \sg_3 < \sg_4 > \sg_5$. This approach will 
be described in a forth coming paper.

\end{document}